\documentclass[preprint,12pt]{elsarticle}

\usepackage{amssymb}
\usepackage{amsthm}

\usepackage{amsmath}

\usepackage{color}

\usepackage{tikz}
\newtheorem{theorem}{Theorem}[section]
\newtheorem{lemma}[theorem]{Lemma}
\newtheorem{proposition}[theorem]{Proposition}
\newtheorem{corollary}[theorem]{Corollary}

\theoremstyle{definition}
\newtheorem{example}{Example}

\journal{Fuzzy Sets and Systems}

\begin{document}

\begin{frontmatter}



\title{Tropical implementation of\\ the Analytical Hierarchy Process decision method}
\author{Nikolai Krivulin\corref{cork}\fnref{fnk}}

\author{Serge{\u{\i}} Sergeev\corref{cors}\fnref{fns}}

\cortext[cork]{Saint Petersburg State University, 7/9 Universitetskaya nab., Saint Petersburg, 199034, Russia. Corresponding author. Email: nkk@math.spbu.ru}

\cortext[cors]{University of Birmingham, School of Mathematics, Edgbaston B15 2TT, UK. Email: s.sergeev@bham.ac.uk}


%\author[label1]{Nikolai Krivulin}

%\address[label 1]{}

\begin{abstract}
We apply methods and techniques of tropical optimization to develop a new theoretical and computational framework for the implementation of the Analytic Hierarchy Process in multi-criteria problems of rating alternatives from pairwise comparison data. In this framework, we first consider the minimax Chebyshev approximation of pairwise comparison matrices by consistent matrices in the logarithmic scale. Recasting this approximation problem as a problem of tropical pseudo-quadratic programming, we then write out a closed-form solution to it. This solution might be either a unique score vector (up to a positive factor) or a set of different score vectors. To handle the problem when the solution is not unique, we develop tropical optimization techniques of maximizing and minimizing the Hilbert seminorm to find those vectors from the solution set that are the most and least differentiating between the alternatives with the highest and lowest scores, and thus are well representative of the entire solution set.
\end{abstract}


\end{frontmatter}

% \linenumbers

%\section{}
%\label{}

\section{Introduction}
\label{s:intro}

Tropical (idempotent) mathematics, which deals with the theory and applications of algebraic systems with idempotent operations \cite{Golan2003Semirings,Heidergott2006Maxplus,Mceneaney2006Maxplus,Butkovic2010Maxlinear,Maclagan2015Introduction}, is widely used as a coherent analytical framework to solve problems in engineering, operations research and computer science. Tropical optimization presents an important research domain in this area, focused on optimization problems that are formulated and solved in the tropical mathematics setting.
Methods and techniques of tropical optimization are applied to solve many well-known and new optimization problems in various fields, including decision making \cite{Elsner2004Maxalgebra,Elsner2010Maxalgebra,Gursoy2013Analytic,Tran2013Pairwise,Gavalec2015Decision}. 
%Specifically, in \cite{Gursoy2013Analytic}, the Analytic Hierarchy Process (AHP) method in multi-criteria problems of rating alternatives from pairwise comparison data is investigated from the tropical point of view.

The traditional Analytic Hierarchy Process (AHP) method \cite{Saaty1977Scaling,Saaty1990Analytic,Saaty2013Onthemeasurement} consists of two principal levels of pairwise comparisons: the upper level, where the relative importance of criteria is estimated, and the lower level, where the relative quality of choices is evaluated with respect to each criterion. The final decision is made by combining the rates of all choices computed on the lower level, and the weights of all criteria on the higher level.

The rates of choices with respect to each criterion are found through the rank-one approximation of pairwise comparison matrices, typically by using the principal (Perron) eigenvector methods \cite{Saaty1977Scaling,Saaty1990Analytic,Saaty1984Comparison,Saaty2013Onthemeasurement}, and, sometimes, by other techniques, including the least squares or the logarithmic least squares methods \cite{Saaty1984Comparison,Chu1998Ontheoptimal,Saaty1984Comparison,Barzilai1997Deriving,Farkas2003Consistency}. The weights of criteria can be evaluated in the same manner from pairwise comparisons or obtained in a different way.

More specifically, assume that there are $m$ criteria and $n$ choices. Given a pairwise comparison matrix $A_{k}=(a_{ij}^{(k)})$ of order $n$, and a weight $w_{k}$ for each criterion $k$, the vector of the priorities of all choices $x=(x_{i})$ is calculated as
\begin{equation}
x
=
\sum_{k=1}^{m}w_{k}y_{k},
\label{e:AHP} 
\end{equation}
where $y_{k}$ is the vector of rates, obtained from $A_{k}$.

Since the first appearance of the AHP decision approach, several new implementations of AHP have been developed using various mathematical techniques, including fuzzy AHP \cite{Laarhoven1983Fuzzy,Kubler20163Sateoftheart}, interval AHP \cite{Ahn2017Analytic}, and some others. Specifically, in \cite{Gursoy2013Analytic} a variant of AHP based on tropical mathematics is proposed, using the matrix approximation in terms of the minimization of 
the maximum relative error \cite{Elsner2004Maxalgebra,Elsner2010Maxalgebra}. Based on the observation of \cite{Elsner2010Maxalgebra} that this minimum is attained by any tropical subeigenvector, the paper \cite{Gursoy2013Analytic} suggests to seek a common subeigenvector of pairwise comparison matrices $A^{(k)}$ for all criteria $k$ (giving a number of conditions for existence of that common subeigenvector) or, alternatively, to seek a Pareto optimal solution.

In this paper, we develop a new theoretical and computational framework for the implementation of AHP, based on the tropical optimization techniques proposed in \cite{Krivulin2015Rating,Krivulin2016Using}. The new AHP method, which we offer and investigate below, aims to find a rank-one matrix that should be the closest, in the sense of the maximum of weighted log-Chebyshev distances, to all the pairwise comparison matrices corresponding to different criteria. As we argue in Section~\ref{s:AHP-intro} the vector of priorities $x=(x_{i})$ is a solution of the following optimization problem
\begin{equation}
\begin{aligned}
&
\min_{x}
&&
\max_{i,j=1}^{n}\max_{k=1}^{m}(w_{k}a_{ij}^{(k)})x_{j}/x_{i}.
\end{aligned}
\label{P-minxwkaijkxjxi}
\end{equation}

The new approach is a modification of the direct weighted sum calculation in the basic AHP scheme \eqref{e:AHP}. In this modification, the weights of criteria are incorporated into the objective function of the optimization problem solved at the lower level of AHP. We thus minimize a max-linear combination of lower level objective functions (representing the consistency of judgment with respect to various criteria), in which they are multiplied by their proper weights taken from the higher level of AHP. 

The solution obtained as a result of the tropical AHP method is, in general, non-unique. However, inconsistency and ambiguity of determining the pairwise preferences are inherent to the process of forming the pairwise comparison matrices, and therefore it seems quite natural that the method ends up with a set of solution vectors rather than with just one vector as in the ordinary AHP.

To make the non-unique result tractable and useful for practice, we focus on two kinds of solutions that can be considered, in some sense, as the best and worst solutions. The solution set is characterized by vectors that are the most and least differentiating between the choices with the highest and lowest priorities. See Section~\ref{s:AHP-intro} and problem formulations \eqref{P-maxxmaxximax1xj}, \eqref{P-minxmaxximax1xj}.

Our next goal is then to describe the sets of most and least differentiating vectors in the framework of tropical mathematics, using some basic facts about the tropical linear algebra and the 
tropical pseudoquadratic optimization given in Section~\ref{s:elements}. Such description is obtained in Subsections~\ref{ss:maximization} and~\ref{ss:minimization}, and its geometric sense is then illustrated in Subsection~\ref{ss:geometry}. 

Note that the solution set of \eqref{P-minxwkaijkxjxi} is a tropical convex cone: a subset of nonnegative orthant, which is closed under multiplication by a scalar factor and componentwise maximum of two vectors. The problem of minimizing the Hilbert semidistance of a point from a tropical convex cone or, more generally, from an idempotent semimodule, was considered in \cite{Cohen2004Duality} and \cite{Akian2011Best}, where it was shown to be in a close relation with the tropical Hahn-Banach theorem. Below we will consider a special case of this problem where a simple algebraic description of the whole solution set will be available: see Subsection~\ref{ss:minimization}. 

The problem of finding the maximum of Hilbert seminorm \eqref{P-maxxmaxximax1xj} over the tropical column span of normal matrices\footnote{A square matrix is tropically normal if it has the diagonal entries equal to and the off-diagonal entries less than or equal to the tropical unit \cite{Butkovic2010Maxlinear}.} -- but without describing the whole solution set -- was solved in \cite{Delapuente2013Ontropical}. A complete solution of this problem that does not assume the normality or any other property of the matrix will be given in Subsection~\ref{ss:maximization}.

Examples of application of our tropical AHP framework are given in Section~\ref{s:AHP}, where we consider two multi-criteria decision problems from \cite{Saaty1977Scaling}. This is followed by a discussion of the differences between our approach and the traditional implementation of AHP, and some possibilities for further research. 

The main contribution of this paper is a new tropical framework for implementation of AHP. This implementation is based on complete description of the most and least differentiating vectors. In mathematical terms, this description results from a complete closed-form solution of two optimization problems over the tropical column span of a nonnegative matrix. We refine these solutions and provide a new geometric interpretation, which consists in maximization and minimization of the Hilbert seminorm over the tropical column span of a nonnegative matrix.

Let us mention that the new tropical implementation of AHP presented here has been outlined in our short conference paper \cite{Krivulin2017Tropicaloptimization}. That conference paper also presents an application of our new scheme to one of the examples described in Section 5 of the present paper. It also contains the formulation of some of the basic facts on which the new method is based, but without any proofs.

\section{Minimax approximation based AHP}
\label{s:AHP-intro}

In this section, we describe a new approach to develop an AHP decision scheme that is based on the rank-one log-Chebyshev approximation of pairwise comparison matrices. We also call it the tropical implementation of AHP since the tropical linearity is essential for closed-form description of solutions at each step.

\subsection{Log-Chebyshev approximation of pairwise comparison matrices}

Consider the problem of evaluating the rates of $n$ choices from pairwise comparisons of these choices. The outcome of these comparisons is described by a square symmetrically reciprocal matrix $A=(a_{ij})$, where $a_{ij}$ specifies the relative priority of choice $i$ over $j$, and satisfies the condition $a_{ij}=1/a_{ji}>0$ for all $i,j$. The pairwise comparison matrix $A$ is called consistent if its entries are transitive, that is, if they satisfy the equality $a_{ij}=a_{ik}a_{kj}$ for all $i,j,k$.

For each consistent matrix $A$, there is a positive vector $x=(x_{i})$ whose elements completely determine the entries of $A$ by the relation $a_{ij}=x_{i}/x_{j}$, which, in particular, means that $A$ is a matrix of unit rank. Provided that the matrix $A$ is consistent, its corresponding vector $x$, which can be readily obtained from $A$, directly represents, up to a positive factor, the individual preferences of choices in question.

Since the pairwise comparison matrices, encountered in practice, are generally inconsistent, the solution usually involves approximating these matrices by consistent matrices. The approximation with the principal (Perron) eigenvector as well as the least squares or the logarithmic least squares approximation are often used as solution approaches.

Consider another approach \cite{Krivulin2015Rating,Krivulin2016Using}, which is based on the approximation of a pairwise comparison matrix $A=(a_{ij})$ by a consistent matrix $X=(x_{ij})$ in the log-Chebyshev sense, where the approximation error is measured with the Chebyshev metric on the logarithmic scale. Taking into account that both matrices $A$ and $X$ have positive entries, and that the logarithmic function (to the base more than one) is monotonically increasing, the error can be written as
\begin{equation*}
\max_{i,j=1}^{n}|\log a_{ij}-\log x_{ij}|
=
\log\max_{i,j=1}^{n}\max\{a_{ij}/x_{ij},x_{ij}/a_{ij}\}.
\end{equation*}

Observing that the minimization of the logarithm is equivalent to the minimization of its argument, and that $a_{ij}=1/a_{ji}$ and $x_{ij}=x_{i}/x_{j}$, we replace the last logarithm by $\max_{i,j=1}^{n}\max\{a_{ij}/x_{ij},x_{ij}/a_{ij}\}=\max_{i,j=1}^{n}a_{ij}x_{j}/x_{i}$. This reduces solving the approximation problem to solving, with respect to the unknown vector of priorities $x=(x_{i})$, the optimization problem 
\begin{equation}
\min_{x}\ (\max_{i,j=1}^{n}a_{ij}x_{j}/x_{i}).
\label{P-minxaijxjxi}
\end{equation}

Note that problem \eqref{P-minxaijxjxi} is equivalent to that arising in the approximation by minimizing the maximum relative error in \cite{Elsner2004Maxalgebra,Elsner2010Maxalgebra}.

\subsection{Weighted approximation under several criteria}

Suppose the priorities of choices are evaluated based on pairwise comparisons according to $m$ criteria, each having a given weight. For each criterion $k$, we denote the pairwise comparison matrix by $A_{k}=(a_{ij}^{(k)})$ and the positive weight by $w_{k}$. To determine the priority vector $x=(x_{i})$, we minimize the maximum of the functions $\max_{i,j=1}^{n}a_{ij}^{(k)}x_{j}/x_{i}$, taken with the weights $w_{k}$ for all $k$. That is, we pose the following problem:
\begin{equation}
\min_{x}\ (\max_{k=1}^{m}w_{k}(\max_{i,j=1}^{n}a_{ij}^{(k)}x_{j}/x_{i}))
=
\min_{x}\ (\max_{i,j=1}^{n}\max_{k=1}^{m}(w_{k}a_{ij}^{(k)})x_{j}/x_{i}).
\label{P-weighted-approx}
\end{equation}

Introducing the matrix $B=(b_{ij})$ with the entries
\begin{equation}
b_{ij}
=
\max_{k=1}^{m}w_{k}a_{ij}^{(k)},
\label{E-bij-maxwkaijk}
\end{equation}
we can reduce \eqref{P-weighted-approx} to a problem in the form of \eqref{P-minxaijxjxi}, with $B$ instead of $A$. The solution of \eqref{P-weighted-approx} can be considered as a modification of the basic AHP scheme, in which the log-Chebyshev approximation is used instead of the principal eigenvector method, and the weights of criteria are incorporated into the lower-level evaluation of choices.

\subsection{Most and least differentiating priority vectors}
\label{ss:mostandleast}

In general, two priority vectors that solve problem \eqref{P-minxaijxjxi} or \eqref{P-weighted-approx} cannot necesarily be obtained from one another by means of multiplication by a positive factor. That is, solution of~\eqref{P-minxaijxjxi} or \eqref{P-weighted-approx} can be essentially non-unique, in general. Below, we develop an approach in which the entire solution is ``represented'' by two vectors, which can be considered, in some sense, as the best and worst solutions.

Assume that problem \eqref{P-minxaijxjxi} or \eqref{P-weighted-approx} has a set $S$ of solutions $x=(x_{i})$ rather than a unique one (up to a scalar factor multiplication). Since the main purpose of evaluating priorities is to differentiate between choices, we find the solutions that are the most and least differentiating between the choices with the highest and lowest priorities. The calculation of the most and least differentiating vectors involves determining the exact bounds for the contrast ratio
\begin{equation*}
(\max_{i=1}^{n}x_{i})/(\min_{j=1}^{n}x_{j})
=
(\max_{i=1}^{n}x_{i})(\max_{j=1}^{n}(1/x_{j}))
\end{equation*}
aiming to find the vectors $x$, which solve the problem of the Hilbert (span, range) seminorm maximization
\begin{equation}
\max_{x\in S}\ (\max_{i=1}^{n}x_{i})(\max_{j=1}^{n}(1/x_{j})),
\label{P-maxxmaxximax1xj}
\end{equation}
and the problem of the Hilbert seminorm minimization
\begin{equation}
\min_{x\in S}\ (\max_{i=1}^{n}x_{i})(\max_{j=1}^{n}(1/x_{j})).
\label{P-minxmaxximax1xj}
\end{equation}

The Hilbert seminorm of $x$ in the logarithmic scale is actually defined as $\log ((\max_{i=1}^{n}x_{i})(\max_{j=1}^{n}(1/x_{j})))$, but, for the sake of optimization, the logarithm can be omitted, since it is a monotone function.

In subsequent sections, we will treat problems \eqref{P-minxaijxjxi}, \eqref{P-maxxmaxximax1xj} and \eqref{P-minxmaxximax1xj} in terms of tropical mathematics, and give direct and explicit solutions, which are ready for immediate computation.

\section{Tropical linear algebra and tropical pseudo-quadratic programming}
\label{s:elements}

We start with a brief overview of basic definitions and notation of tropical (idempotent) algebra to provide a formal framework for describing tropical optimization techniques, used below in the development of tropical implementation of AHP. Further details on tropical mathematics can be found, e.g., in \cite{Golan2003Semirings,Heidergott2006Maxplus,Mceneaney2006Maxplus,Butkovic2010Maxlinear,Maclagan2015Introduction}.

\subsection{Tropical linear algebra}
\label{ss:maxalg}

We consider the set of non-negative reals $R_{+}$ equipped with two operations: addition $\oplus$, defined as $\max$, and multiplication $\otimes$, defined as the usual multiplication, with the neutral elements: zero $0$ and one $1$. Addition $\oplus$ is idempotent since $x\oplus x=\max(x,x)=x$ for each $x\in R_{+}$, multiplication distributes over addition $\oplus$ and is invertible, providing each $x>0$ with the inverse $x^{-1}$ such that $x\otimes x^{-1}=xx^{-1}=1$. The idempotent algebraic system $(R_{+},0,1,\oplus,\otimes)$ is commonly referred to as the tropical algebra or max-algebra, and denoted by $R_{\max}$.

In the tropical algebra, both addition and multiplication are monotone in their arguments, which means that the inequality $x\leq y$ implies the inequalities $x\oplus z\leq y\oplus z$ and $x\otimes z\leq y\otimes z$ for any 
$x,y,z\in R_{+}$. Moreover, the inequality $x\oplus y\leq z$ is equivalent to the system of inequalities $x\leq z$ and $y\leq z$. The inversion is antitone in the sense that if $x\leq y$ for some $x,y>0$, then $x^{-1}\geq y^{-1}$.

Since the multiplication $\otimes$ defined in $R_{\max}$ coincides with the standard arithmetic multiplication, the power notation $x^{p}$ has the usual interpretation for all $x>0$ and rational $p$. In what follows, the multiplication sign $\otimes$ is omitted for the sake of brevity. 

The scalar tropical algebra is routinely extended to the set of non-negative matrices over $R_{+}$ with the matrix operations defined by the conventional rules, where the scalar addition and multiplication are replaced by the operations $\oplus$ and $\otimes$. This is referred to as the tropical linear algebra.
As usual, a matrix with all zero entries is called the zero matrix.

The multiplicative conjugate transpose (or simply the conjugate transpose) of a nonzero $(m\times n)$-matrix $A=(a_{ij})$ is the $(n\times m)$-matrix $A^{-}=(a_{ij}^{-})$ with the entries $a_{ij}^{-}=a_{ji}^{-1}$ if $a_{ji}\ne0$, and $a_{ij}^{-}=0$ otherwise.

Any matrix that consists of one column is a column vector. The column vector with all zero entries is the zero vector $0$. The column vector with all entries equal to $1$ is denoted by $1$.

The conjugate transpose of a nonzero column vector $x=(x_{i})$ is the row vector $x^{-}=(x_{i}^{-})$, where $x_{i}^{-}=x_{i}^{-1}$ if $x_{i}\ne0$, and $x_{i}^{-}=0$ otherwise.  

The monotone properties of the scalar operations $\oplus$ and $\otimes$ are readily carried over to the matrix and vector operations, where the relations are understood entrywise. Specifically, for all positive matrices $A$ and $B$ such that $A\leq B$, the conjugate transposition satisfies $A^{-}\geq B^{-}$.

An $m$-vector $b$ is linearly dependent on $m$-vectors $a_{1},\ldots,a_{n}$ if there are non-negative numbers $x_{1},\ldots,x_{n}$ such that $b=x_{1}a_{1}\oplus\cdots\oplus x_{n}a_{n}$. Specifically, a vector $b$ is collinear with a vector $a$, if $b=xa$ for some scalar $x$.

A positive matrix $A$ is of rank $1$ if and only if $A=xy^{T}$, where $x$ and $y$ are positive column vectors. A matrix $A$ that satisfies the condition $A^{-}=A$ is called symmetrically reciprocal (or simply reciprocal). A reciprocal matrix $A$ is of rank $1$ if and only if $A=xx^{-}$, where $x$ is a positive column vector.

For any square matrix $A$ and integer $p>0$, the tropical (or max-algebraic) power notation is routinely defined by the inductive rule $A^{p}=A^{p-1}A$, where $A^{0}=I$ is the usual identity matrix.

The tropical (max-algebraic) spectral radius of an $(n\times n)$-matrix $A=(a_{ij})$ is computed as the maximum cycle geometric mean of the matrix entries, which is given by
\begin{equation}
\lambda
=
\bigoplus_{1\leq k\leq n}\bigoplus_{1\leq i_{1},\ldots,i_{k}\leq n}(a_{i_{1}i_{2}}a_{i_{2}i_{3}}\cdots a_{i_{k}i_{1}})^{1/k}
=
\mathop\mathrm{tr}A
\oplus\cdots\oplus
\mathop\mathrm{tr}\nolimits^{1/n}(A^{n}).
\label{E-lambda-ai1i2ai2i3aiki1}
\end{equation}

We also use the function, which maps the matrix $A$ onto the scalar
\begin{equation*}
\mathop\mathrm{Tr}(A)
=
\bigoplus_{m=1}^{n}\mathop\mathrm{tr}A^{m}
=
\mathop\mathrm{tr}A
\oplus\cdots\oplus
\mathop\mathrm{tr}A^{n},
\end{equation*}
and note that if $\lambda>0$ then the inequality $\mathop\mathrm{Tr}(\lambda^{-1}A)\leq1$ holds.

Provided that $\mathop\mathrm{Tr}(A)\leq1$, the asterate operator (the Kleene star) yields the matrix
\begin{equation*}
A^{\ast}
=
\bigoplus_{m=0}^{n-1}A^{m}
=
I\oplus A\oplus\cdots\oplus A^{n-1}.
\end{equation*}

Finally, we consider the problem to find positive vectors $x$ that solve the inequality
\begin{equation}
Ax\leq x.
\label{I-Axx}
\end{equation}

The next result obtained in \cite{Krivulin2015Extremal} offers a complete solution to this inequality (see also \cite{Butkovic2010Maxlinear} and references therein). 
\begin{lemma}
\label{L-Axx}
For any square matrix $A$, the following statements hold:
\begin{enumerate}
\item If $\mathop\mathrm{Tr}(A)\leq1$, then all positive solutions to \eqref{I-Axx} are given by $x=A^{\ast}u$, where $u$ is any positive vector.
\item If $\mathop\mathrm{Tr}(A)>1$, then there is no positive solution.
\end{enumerate}
\end{lemma}

Below, we use the algebraic preliminaries introduced above to describe tropical optimization problems and their solutions.

\subsection{Tropical pseudo-quadratic programming}
\label{ss:pseudoquad}

In this section, we consider optimization problems, which are formulated and solved in the tropical algebra setting, to provide the basis for our tropical implementation of AHP. 

First, assume that, given a non-negative $(n\times n)$-matrix $A$, we need to find positive $n$-vectors $x$ that solve the problem
\begin{equation}
\begin{aligned}
&
\min_{x}
&&
x^{-}Ax.
\end{aligned}
\label{P-minxxAx}
\end{equation}

A complete, direct solution to the problem was obtained in \cite{Krivulin2015Extremal} (see also \cite{Butkovic2010Maxlinear} and references therein).
\begin{lemma}
\label{L-minxxAx}
Let $A$ be a matrix with tropical spectral radius $\lambda>0$. Then, the optimal value in problem \eqref{P-minxxAx} is equal to $\lambda$, and all positive solutions are given by
\begin{equation*}
x
=
(\lambda^{-1}A)^{\ast}u,
\quad
u>0.
\end{equation*}
\end{lemma}

We now suppose that $A_{1},\ldots,A_{m}$ are given non-negative $(n\times n)$-matrices, and $w_{1},\ldots,w_{m}$ are given positive numbers. The problem is to find positive $n$-vectors $x$ that attain the minimum in
\begin{equation}
\begin{aligned}
&
\min_{x}
&&
\bigoplus_{k=1}^{m}w_{k}x^{-}A_{k}x.
\end{aligned}
\label{P-minxwkxAkx}
\end{equation}

As a direct consequence of the previous result, we have the following solution \cite{Krivulin2016Using}.
\begin{corollary}
\label{C-minxwkxAkx}
Let $A_{1},\ldots,A_{m}$ be non-negative matrices and $w_{1},\ldots,w_{m}$ be positive numbers such that the matrix $B=w_{m}A_{m}\oplus\cdots\oplus w_{m}A_{m}$ has the tropical spectral radius $\mu>0$.

Then, the minimum value in \eqref{P-minxwkxAkx} is equal to $\mu$, and all positive solutions are given by
\begin{equation*}
x
=
(\mu^{-1}B)^{\ast}u,
\quad
u>0.
\end{equation*}
\end{corollary}

Furthermore, given a matrix and two vectors, we examine two problems, which take the form of an unconstrained maximization and a constrained minimization problems. Let $A=(a_{j})$ be a non-negative $(m\times n)$-matrix with columns $a_{j}=(a_{ij})$, and $p=(p_{i})$ be an $m$-vector and $q=(q_{j})$ an $n$-vector. Consider the problem to find positive vectors $x=(x_{j})$ that attain the maximum
\begin{equation}
\begin{aligned}
&
\max_{x}
&&
q^{-}x(Ax)^{-}p.
\end{aligned}
\label{P-maxxqxAxp}
\end{equation}

The next result obtained in \cite{Krivulin2016Maximization,Krivulin2017Algebraic} offers a complete solution to problem \eqref{P-maxxqxAxp} under fairly general conditions.
\begin{proposition}
\label{L-maxxqxAxp}
Let $A$ be a positive matrix, $p$ be a nonzero vector, $q$ be a positive vector, and $\Delta=q^{-}A^{-}p$. Let $A_{lk}$ be the matrix obtained from $A$ by keeping the entry $a_{lk}$ for some indices $l$ and $k$, and replacing the other entries by zero.

Then, the optimal value in problem \eqref{P-maxxqxAxp} is equal to $\Delta$, and all positive solutions are given by the conditions
\begin{equation*}
x
=
(I\oplus A_{lk}^{-}A)u,
\quad
u>0,
\end{equation*}
for all indices $k$ and $l$ defined by the condition
\begin{equation*}
k
=
\arg\max_{j=1}^{m}q_{j}^{-1}a_{j}^{-}p,
\qquad
l
=
\arg\max_{i=1}^{n}a_{ik}^{-1}p_{i}.
\end{equation*}
\end{proposition} 

Finally, suppose that we need to find positive solutions of the constrained minimization problem
\begin{equation}
\begin{aligned}
&
\min_{x}
&&
q^{-}xx^{-}p,
\\
&&&
Ax
\leq
x.
\end{aligned}
\label{P-minxqxxp-Axx}
\end{equation} 

A complete solution was obtained in \cite{Krivulin2017Tropical}, and it can be described as follows.
\begin{proposition}
\label{L-minxqxxp-Axx}
Let $A$ be a matrix such that $\mathop\mathrm{Tr}(A)\leq1$, $p$ be a nonzero vector, $q$ be a positive vector, and $\delta=q^{-}A^{\ast}p$. 

Then, the optimal value in problem \eqref{P-minxqxxp-Axx} is equal to $\delta$, and all regular solutions are given by
\begin{equation*}
x
=
(\delta^{-1}pq^{-}\oplus A)^{\ast}u,
\quad
u>0.
\end{equation*}
\end{proposition}

\section{Maximizing and minimizing the Hilbert seminorm}
\label{s:math}

In this section, we further simplify and refine the solutions of problems \eqref{P-maxxqxAxp} and \eqref{P-minxqxxp-Axx} to make them more appropriate for use in the tropical AHP below. We give a new more simple proof for the statement of Proposition~\ref{L-maxxqxAxp} for the maximization problem, and then obtain a compact closed-form solution of the special cases of maximization problem \eqref{P-maxxqxAxp} and minimization problem \eqref{P-minxqxxp-Axx}. The geometric sense of these problems and their solutions is then illustrated in Subsection~\ref{ss:geometry}.

\subsection{Maximization problem}
\label{ss:maximization}

We start with the derivation of a new more compact solution to the maximization problem given by \eqref{P-maxxqxAxp}. 

First note that, without loss of generality, we can consider the vector $p$ to be positive. If $p$ has zero components, then these components can be eliminated together with the corresponding rows of the matrix $A$, which yields an equivalent problem in the form of \eqref{P-maxxqxAxp} with a positive vector $p$.

The next statement establishes the maximum value of the objective function in problem \eqref{P-maxxqxAxp} and describes all vectors $x$ that yield this minimum. The derivation of the upper bound for the objective function is taken from \cite{Krivulin2016Maximization} and included in the proof below for the sake of completeness.  
 
\begin{theorem}
\label{T-maxxqxAxp}
Let $A$ be a positive matrix, $p$ be a nonzero vector, $q$ be a positive vector, and $\Delta=q^{-}A^{-}p$. Then, the optimal value in problem \eqref{P-maxxqxAxp} is equal to $\Delta$, and all positive solutions are given by
\begin{equation}
\bigoplus_{j=1}^{n}a_{lj}x_{j}
=
a_{lk}x_{k}
\label{E-aljxjalkxk}
\end{equation}
for all indices $k$ and $l$ defined by the condition
\begin{equation}
q_{k}^{-1}a_{lk}^{-1}p_{l}
=
\Delta.
\label{E-qkalkplqAp}
\end{equation}
\end{theorem}
\begin{proof}
We take the obvious inequality $xx^{-}\geq I$, which is valid for all positive vectors $x$. Multiplying it from the left by $A$, we obtain the inequality $Axx^{-}\geq A$. Since both sides of this inequality are matrices with positive entries, we further have $(Axx^{-})^{-}\leq A^{-}$ by conjugate transposing. The latter inequality is the same as $x(Ax)^{-}\leq A^{-}$, which we multiply by $q^{-}$ on the left and by $p$ on the right to obtain $q^{-}x(Ax)^{-}p\leq q^{-}A^{-}p=\Delta$. Thus $q^{-}A^{-}p$ is an upper bound for $q^{-}x(Ax)^{-}p$.

Let us show that there exists $x$ such that $q^{-}x(Ax)^{-}p\geq\Delta=q^{-}A^{-}p$. We define indices $k$ and $l$, and take a vector $x$ according to the conditions
\begin{equation*}
q^{-}A^{-}p
=
\bigoplus_{i=1}^{n}\bigoplus_{j=1}^{m}q_{i}^{-1}a_{ji}^{-1}p_{j}
=
q_{k}^{-1}a_{lk}^{-1}p_{l},
\qquad
\bigoplus_{j=1}^{n}a_{lj}x_{j}
=
a_{lk}x_{k}.
\end{equation*}

With this vector $x$, we obtain
\begin{equation*} 
q^{-}x(Ax)^{-}p
\geq
q_{k}^{-1}x_{k}(a_{lk}x_{k})^{-1}p_{l}
=
q_{k}^{-1}a_{lk}^{-1}p_{l}
=
q^{-}A^{-}p.
\end{equation*}

With the opposite inequality, we have $q^{-}x(Ax)^{-}p=q^{-}A^{-}p$, which means that $q^{-}A^{-}p$ is a strict upper bound, and thus the maximum in problem \eqref{P-maxxqxAxp}.

Let us now take an arbitrary solution $x$ of \eqref{P-maxxqxAxp}, and then verify that $x$ satisfies \eqref{E-aljxjalkxk} under condition \eqref{E-qkalkplqAp}. First note, that, for such $x$, we have $q^{-}x(Ax)^{-}p=q^{-}A^{-}p$. Let indices $s$ and $t$ be defined by the conditions
\begin{equation*}
q^{-}x=q_{s}^{-1}x_s,
\qquad
(Ax)^{-}p=(Ax)^{-1}_{t}p_{t}.
\end{equation*}

Then, we can write the following chain of equalities and inequalities:
\begin{equation*}
q^{-}x(Ax)^{-}p
=
q_{s}^{-1}x_{s}(Ax)^{-1}_{t}p_{t}
\leq
q_{s}^{-1}x_{s}a_{ts}^{-1}x_{s}^{-1}p_{t}
=
q_{s}^{-1}a_{ts}^{-1}p_{t}
\leq
q^{-}A^{-}p.
\end{equation*}

However, since $q^{-}A^{-}p=q^{-}x(Ax)^{-}p$, both inequalities in this chain turn into equalities. As a result, we have
\begin{equation*}
q_{s}^{-1}a_{ts}^{-1}p_{t}
=
q^{-}A^{-}p,
\qquad
(Ax)_{t}
=
a_{ts}x_{s},
\end{equation*}
which means that \eqref{E-aljxjalkxk} and \eqref{E-qkalkplqAp} are satisfied with $k=s$ and $l=t$ .
\end{proof}

The entrywise positivity of the matrix $A$ is important for the above proof, particularly in inverting the inequality $Axx^{-}\geq A$. Indeed, if we take $A=I$ then $Axx^{-}\geq A$ does not imply $(Axx^{-})^{-}\leq A^{-}$, where $A^{-}=A=I$. If we define $A^{-}$ as a matrix with $+\infty$ entries then $q^{-}A^{-}p$ may become $+\infty$: a trivial bound, which is never attained.

We now examine a special case of problem \eqref{P-maxxqxAxp} that arises in the tropical implementation of AHP and as the problem of Hilbert seminorm maximization. We set $q^{-}=1^{T}A$ and $p=1$ in \eqref{P-maxxqxAxp}, and consider the problem
\begin{equation}
\begin{aligned}
&
\max_{x}
&&
1^{T}Ax(Ax)^{-}1.
\end{aligned}
\label{P-maxx1AxAx1}
\end{equation}

A complete solution to this problem is formulated as follows.
\begin{corollary}
\label{C-maxx1AxAx1}
Let $A$ be a positive matrix, and $\Delta=1^{T}AA^{-}1$. Let $A_{lk}$ be the matrix obtained from $A$ by keeping the entry $a_{lk}$ for some indices $l$ and $k$, and replacing the other entries by zero.

Then, the optimal value in problem \eqref{P-maxx1AxAx1} is equal to $\Delta$, and all positive solutions are given by
\begin{equation*}
x
=
(I\oplus A_{lk}^{-}A)u,
\qquad
u>0,
\end{equation*}
for all indices $k$ and $l$ defined by the condition
\begin{equation*}
1^{T}a_{k}a_{lk}^{-1}
=
\Delta.
\end{equation*}
\end{corollary}
\begin{proof}
We apply Theorem~\ref{T-maxxqxAxp} with $q^{-}=1^{T}A$ and $p=1$ to represent the maximum in the problem as $q^{-}A^{-}p=1^{T}AA^{-}1$, and the left-hand side of the condition at \eqref{E-qkalkplqAp} as $q_{k}^{-1}a_{lk}^{-1}p_{l}=1^{T}a_{k}a_{lk}^{-1}$. 

Furthermore, we note that the equality \eqref{E-aljxjalkxk} does not include the vectors $p$ and $q$, and thus remains unchanged. To represent the set of solutions in a compact vector form, we multiply both sides of the equality by $a_{lk}^{-1}$. Furthermore, we introduce a positive $n$-vector of parameters $u=(u_{j})$ and rewrite this equality in a parametric form using the scalar equalities
\begin{equation*}
x_{k}
=
\bigoplus_{j=1}^{n}a_{lk}^{-1}a_{lj}u_{j},
\qquad
x_{i}
=
u_{i},
\qquad
i\ne k.
\end{equation*}

We denote by $A_{lk}$ the matrix obtained from $A$ by setting all entries other than $a_{lk}$ to zero. With this matrix, we represent the scalar equalities in the vector form
\begin{equation*}
x
=
(I\oplus A_{lk}^{-}A)u,
\qquad
u>0,
\end{equation*}
which completes the proof.
\end{proof}

\subsection{Minimization problem}
\label{ss:minimization}

We now consider a constrained minimization problem that we use in the tropical implementation of AHP, and solve it by reducing to problem \eqref{P-minxqxxp-Axx}. Suppose that, given a matrix $A$ with spectral radius $\lambda>0$, the problem is to find positive vectors $x$ that yield the minimum
\begin{equation}
\begin{aligned}
&
\min_{x}
&&
1^{T}xx^{-}1,
\\
&&&
x
=
(\lambda^{-1}A)^{\ast}u,
\quad
u>0.
\end{aligned}
\label{P-minx1xx1-xlambdaAastu}
\end{equation}

The next result offers a complete solution to the problem.
\begin{corollary}
\label{C-minx1xx1-xlambdaAastu}
Let $A$ be a matrix with spectral radius $\lambda>0$, and $\delta=1^{T}(\lambda^{-1}A)^{\ast}1$. 

Then, the optimal value in problem \eqref{P-minx1xx1-xlambdaAastu} is equal to $\delta$, and all positive solutions are given by
\begin{equation*}
x
=
(\delta^{-1}11^{T}\oplus\lambda^{-1}A)^{\ast}u,
\qquad
u>0.
\end{equation*}
\end{corollary}
\begin{proof} 
We consider the equality $x=(\lambda^{-1}A)^{\ast}u$, and note that by Lemma~\ref{L-Axx}, this equality means that $x$ is determined by the inequality $\lambda^{-1}Ax\leq x$.

Observing that $\mathop\mathrm{Tr}(\lambda^{-1}A)\leq1$, we apply Proposition~\ref{L-minxqxxp-Axx}, where $A$ is replaced by $\lambda^{-1}A$ and both $q$ and $p$ by $1$, and thus complete the proof.
\end{proof}

\subsection{Geometric interpretation}
\label{ss:geometry}

Problems \eqref{P-maxx1AxAx1} and \eqref{P-minx1xx1-xlambdaAastu} consist in maximizing and minimizing
\begin{equation}
\label{e:Hilb-ratio}
1^{T}xx^{-}1
=
(\max_{i=1}^{n}x_{i})(\max_{j=1}^{n}(1/x_{j})),
\end{equation}
where $x$ belongs to the set $\{Au\colon u\in R_{+}^{n}\}$, $A$ is an $n\times n$ nonnegative matrix (and, more specifically, a Kleene star). That set will be referred to as the tropical column span of $A$ and denoted by $\operatorname{span}(A)$.

The logarithm of ratio \eqref{e:Hilb-ratio} is known as the Hilbert seminorm or range seminorm \cite{Butkovic2010Maxlinear}, or as the Hilbert semidistance between $x$ and $1$ \cite{Cohen2004Duality}. Therefore, problems \eqref{P-maxx1AxAx1} and \eqref{P-minx1xx1-xlambdaAastu} consist in finding the maximum and minimum of the Hilbert seminorm of vectors in the tropical column span, $\operatorname{span}(A)$, or, in other words, finding the maximum and minimum of the Hilbert semidistance between $x$ and $1$. We now give two three-dimensional examples, which illustrate the geometry of the optimization problems under consideration. 

\begin{example}
\label{E-1AuAu1-left}
We start with the following matrix:
\begin{equation*}
A
=
A^{\ast}
=
\begin{pmatrix}
1 & 3/4 & 1/2
\\
4/3 & 1 & 2/3
\\
2/3 & 1/2 & 1
\end{pmatrix}.
\end{equation*}

The problem of minimizing the Hilbert seminorm over $\operatorname{span}(A)$ is posed as follows: 
\begin{equation*}
\begin{aligned}
&
\min_{x}
&&
1^{T}xx^{-}1,
\\
&&&
x=Au,
\quad
u>0,
\end{aligned}
\end{equation*}
and it is the same as \eqref{P-minx1xx1-xlambdaAastu}.

We solve this problem by applying Corollary~\ref{C-minx1xx1-xlambdaAastu}. We observe that the matrix $A$ has the spectral radius $\lambda=1$, and hence $\lambda^{-1}A=A$. The optimal value of this problem is equal to
\begin{equation*}
\delta
=
1^{T}A1
=
\max_{i,j=1}^{n}a_{ij}
=
4/3.
\end{equation*}

To find the solution set, we successively compute 
\begin{equation*}
\delta^{-1}11^{T}\oplus A
= 
\begin{pmatrix}
1 & 3/4 & 3/4
\\
4/3 & 1 & 3/4
\\
3/4 & 3/4 & 1
\end{pmatrix},
\quad
(\delta^{-1}11^{T}\oplus A)^{\ast}
=
\begin{pmatrix}
1 & 3/4 & 3/4
\\
4/3 & 1 & 1
\\
1 & 3/4 & 1
\end{pmatrix}.
\end{equation*}

As the first two columns of the last matrix are proportional to one another, the solution set can be written as 
\begin{equation*}
x
=
\begin{pmatrix}
1 & 3/4
\\
4/3 & 1
\\
1 & 1
\end{pmatrix}u,
\quad
u>0.
\end{equation*}

The section of this solution set by the plane $\{x\colon x_{3}=1\}$ is the segment between $(1, 4/3)$ and $(3/4, 1)$: see the thick blue segment on Figure~\ref{F-1AuAu1}.
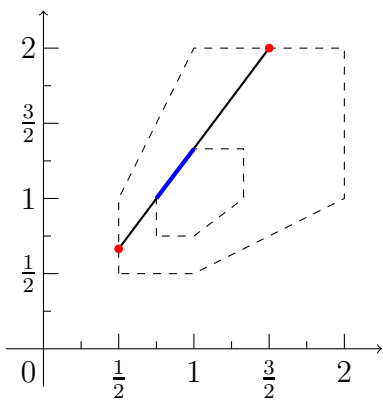
\begin{figure}[ht]
%\begin{tabular}{cc}
\begin{tikzpicture}
\draw[dashed] (1.5,1.5) -- (1.5,2.5) --  (2.5,4.5) -- (4.5,4.5) -- (4.5,2.5) -- (2.5,1.5) -- (1.5,1.5);
\draw[dashed] (2,2) -- (2,2.5) --  (2.5,3.16) -- (3.16,3.16) -- (3.16,2.5) -- (2.5,2) -- (2,2);
\draw[thick] (1.5,1.83) -- (3.5,4.5);
\draw[ultra thick, blue] (2,2.5) -- (2.5,3.16);
\draw[fill,red] (1.5,1.83) circle [radius=0.05];
\draw[fill,red] (3.5,4.5) circle [radius=0.05];
\draw[->] (0,0.5) -- (5, 0.5);
\draw[->] (0.5,0) -- (0.5,5);

%\draw (1,0.5)--
\node at (0.3,0.2) {$0$};

\draw (1,0.5)--(1,0.6);

\draw (1.5,0.5)-- (1.5,0.7);
\node at (1.5,0.1) {$\frac{1}{2}$};

\draw (2,0.5) -- (2,0.6);
%\node at (2,0) {$\frac{3}{4}$};

\draw (2.5,0.5) -- (2.5,0.7);
\node at (2.5,0.2) {$1$};
%\draw (3,0.5) --

\draw (3,0.5)-- (3,0.6);

\draw (3.5,0.5) -- (3.5,0.7);
\node at (3.5,0.1) {$\frac{3}{2}$};

\draw (4,0.5) -- (4,0.6);

\draw (4.5,0.5) -- (4.5,0.7);
\node at (4.5,0.2) {$2$};

\draw (0.5,1) -- (0.6,1);

\draw (0.5,1.5) -- (0.7,1.5);
\node at (0.3,1.5) {$\frac{1}{2}$};

\draw (0.5,2) -- (0.6,2);

%\draw (0.5,1.83) -- (0.7, 1.83);
%\node at (0,1.83) {$\frac{2}{3}$};

\draw (0.5,2.5) -- (0.7,2.5);
\node at (0.3,2.5) {$1$};

\draw (0.5,3) -- (0.6,3);

\draw (0.5,3.5) -- (0.7,3.5);
\node at (0.3,3.5) {$\frac{3}{2}$};

\draw (0.5,4) -- (0.6,4);

\draw (0.5,4.5) -- (0.7,4.5);
\node at (0.3,4.5) {$2$}; 
\end{tikzpicture}
\if{
&
\begin{tikzpicture}
\draw[dashed] (1.5,1.5) -- (1.5,2.5) --  (2.5,4.5) -- (4.5,4.5) -- (4.5,2.5) -- (2.5,1.5) -- (1.5,1.5);
%\draw[dashed] (2,2) -- (2,2.5) --  (2.5,3.16) -- (3.16,3.16) -- (3.16,2.5) -- (2.5,2) -- (2,2);
\draw[thick] (1.5,1.83) -- (3.5,4.5);
%\draw[ultra thick, blue] (2,2.5) -- (2.5,3.16);
\draw[fill,red] (1.5,1.83) circle [radius=0.05];
\draw[fill,red] (3.5,4.5) circle [radius=0.05];

\draw[->] (0,0.5) -- (5, 0.5);
\draw[->] (0.5,0) -- (0.5,5);

%\draw (1,0.5)--
\node at (0.3,0.2) {$0$};

\draw (1,0.5)--(1,0.6);

\draw (1.5,0.5)-- (1.5,0.7);
\node at (1.5,0.1) {$\frac{1}{2}$};

\draw (2,0.5) -- (2,0.6);
%\node at (2,0) {$\frac{3}{4}$};

\draw (2.5,0.5) -- (2.5,0.7);
\node at (2.5,0.2) {$1$};
%\draw (3,0.5) --

\draw (3,0.5)-- (3,0.6);

\draw (3.5,0.5) -- (3.5,0.7);
\node at (3.5,0.1) {$\frac{3}{2}$};

\draw (4,0.5) -- (4,0.6);

\draw (4.5,0.5) -- (4.5,0.7);
\node at (4.5,0.2) {$2$};

\draw (0.5,1) -- (0.6,1);

\draw (0.5,1.5) -- (0.7,1.5);
\node at (0.3,1.5) {$\frac{1}{2}$};

\draw (0.5,2) -- (0.6,2);

%\draw (0.5,1.83) -- (0.7, 1.83);
%\node at (0,1.83) {$\frac{2}{3}$};

\draw (0.5,2.5) -- (0.7,2.5);
\node at (0.3,2.5) {$1$};

\draw (0.5,3) -- (0.6,3);

\draw (0.5,3.5) -- (0.7,3.5);
\node at (0.3,3.5) {$\frac{3}{2}$};

\draw (0.5,4) -- (0.6,4);

\draw (0.5,4.5) -- (0.7,4.5);
\node at (0.3,4.5) {$2$}; 

\end{tikzpicture}

\end{tabular}
}\fi
\caption{Section of the tropical column span of $A$ (black segment) and solution sets to the minimization problem (thick blue segment) and the maximization problem (two ends of black segment, in red) of Example~\ref{E-1AuAu1-left}. Dashed lines are sections of two Hilbert spheres, of radii $\log 4/3$ and $\log 2$.}
\label{F-1AuAu1}
\end{figure}

Let us now consider the problem of maximizing the Hilbert seminorm \eqref{P-maxx1AxAx1}. To find the optimal value by Corollary~\ref{C-maxx1AxAx1}, we calculate
\begin{equation*}
A^{-}
=
\begin{pmatrix}
1 & 3/4 & 3/2
\\
4/3 & 1 & 2
\\
2 & 3/2 & 1
\end{pmatrix},
\qquad
\Delta
=
1^{T}AA^{-}1
=
2. 
\end{equation*}

Furthermore, we calculate 
\begin{equation*}
1^{T}a_{1}
=
4/3,
\qquad
1^{T}a_{2}
=
1^{T}a_{3}
=
1,
\end{equation*}
and then observe that the condition $1^{T}a_{k}a_{lk}^{-1}=\Delta$ is satisfied if either $k=2$ and $l=3$, or $k=3$ and $l=1$.

Let us take $k=2$ and $l=3$. Then, we calculate
\begin{equation*}
A_{32}^{-}
=
\begin{pmatrix}
0 & 0 & 0
\\
0 & 0 & 2
\\
0 & 0 & 0
\end{pmatrix},
\quad
(I\oplus A_{32}^{-}A)
=
\begin{pmatrix}
1 & 0 & 0
\\
4/3 & 1 & 2
\\
0 & 0 & 1
\end{pmatrix}
\end{equation*}
and consider the following solution of \eqref{P-maxx1AxAx1}:
\begin{equation*}
x
=
\begin{pmatrix}
1 & 0 & 0
\\
4/3 & 1 & 2
\\
0 & 0 & 1
\end{pmatrix}u,
\quad
u>0.
\end{equation*}

The corresponding vector in $\operatorname{span}(A)$ is
\begin{equation*}
Ax
=
\begin{pmatrix}
1 & 3/4 & 3/2
\\
4/3 & 1 & 2
\\
2/3 & 1/2 & 1
\end{pmatrix}u.
\end{equation*}

Since all columns of this matrix are proportional to the third column, we take this column to represent the solution. The section of the solution set by the plane $x_{3}=1$ is just one point $(3/2, 2)$.

It is not difficult to verify in the similar way that with $k=3$ and $l=1$, we have the solution
\begin{equation*}
Ax
=
\begin{pmatrix}
1/2
\\
2/3
\\
1
\end{pmatrix}v,
\quad
v>0.
\end{equation*}
which intersects the plane $x_{3}=1$ in the point $(1/2, 2/3)$.

Both solutions are shown on Figure~\ref{F-1AuAu1} (thick red dots).
\end{example}

\begin{example}
\label{E-1AuAu1-right}
Consider problems \eqref{P-maxx1AxAx1} and \eqref{P-minx1xx1-xlambdaAastu} with the matrix
\begin{equation*}
A
=
A^{\ast}
=
\begin{pmatrix}
1 & 3/4 & 1/2
\\
3/4 & 1 & 1/2
\\
1/2 & 1/2 & 1
\end{pmatrix}.
\end{equation*}

To solve problem \eqref{P-minx1xx1-xlambdaAastu} of minimizing the Hilbert seminorm, we note that $\lambda=1$. Next, we find 
\begin{equation*}
\delta
=
1^{T}A1
=
1,
\end{equation*}
which is attained on the ray of points whose all coordinates are equal to each other. The section of this ray by the plane $x_{3}=1$ coincides with the point $(1,1)$ (the thick blue dot on Figure~\ref{F-1AuAu2}). 
\begin{figure}[ht]
%\begin{tabular}{cc}
\begin{tikzpicture}
\draw[dashed] (1.5,1.5) -- (1.5,2.5) --  (2.5,4.5) -- (4.5,4.5) -- (4.5,2.5) -- (2.5,1.5) -- (1.5,1.5);
%\draw[dashed] (2,2) -- (2,2.5) --  (2.5,3.16) -- (3.16,3.16) -- (3.16,2.5) -- (2.5,2) -- (2,2);

%\draw[thick] (1.5,1.83) -- (3.5,4.5);

\path[fill=yellow] (1.5,1.5) -- (1.5,1.83) -- (3.5,4.5) -- (4.5,4.5)-- (4.5,3.5) -- (1.83,1.5) -- (1.5,1.5);
\draw (1.5,1.5) -- (1.5,1.83) -- (3.5,4.5) -- (4.5,4.5)-- (4.5,3.5) -- (1.83,1.5) -- (1.5,1.5);  

\draw[fill,blue] (2.5,2.5) circle [radius=0.05];

\draw[ultra thick, red] (1.83,1.5) -- (1.5,1.5) -- (1.5,1.83);
\draw[ultra thick, red] (3.5,4.5) -- (4.5,4.5) -- (4.5,3.5);

\draw[->] (0,0.5) -- (5, 0.5);
\draw[->] (0.5,0) -- (0.5,5);

%\draw (1,0.5)--
\node at (0.3,0.2) {$0$};

\draw (1,0.5)--(1,0.6);

\draw (1.5,0.5)-- (1.5,0.7);
\node at (1.5,0.1) {$\frac{1}{2}$};

\draw (2,0.5) -- (2,0.6);
%\node at (2,0) {$\frac{3}{4}$};

\draw (2.5,0.5) -- (2.5,0.7);
\node at (2.5,0.2) {$1$};
%\draw (3,0.5) --

\draw (3,0.5)-- (3,0.6);

\draw (3.5,0.5) -- (3.5,0.7);
\node at (3.5,0.1) {$\frac{3}{2}$};

\draw (4,0.5) -- (4,0.6);

\draw (4.5,0.5) -- (4.5,0.7);
\node at (4.5,0.2) {$2$};

\draw (0.5,1) -- (0.6,1);

\draw (0.5,1.5) -- (0.7,1.5);
\node at (0.3,1.5) {$\frac{1}{2}$};

\draw (0.5,2) -- (0.6,2);

%\draw (0.5,1.83) -- (0.7, 1.83);
%\node at (0,1.83) {$\frac{2}{3}$};

\draw (0.5,2.5) -- (0.7,2.5);
\node at (0.3,2.5) {$1$};

\draw (0.5,3) -- (0.6,3);

\draw (0.5,3.5) -- (0.7,3.5);
\node at (0.3,3.5) {$\frac{3}{2}$};

\draw (0.5,4) -- (0.6,4);

\draw (0.5,4.5) -- (0.7,4.5);
\node at (0.3,4.5) {$2$}; 
\end{tikzpicture}
\if{
&
\begin{tikzpicture}
\draw[dashed] (1.5,1.5) -- (1.5,2.5) --  (2.5,4.5) -- (4.5,4.5) -- (4.5,2.5) -- (2.5,1.5) -- (1.5,1.5);
%\draw[dashed] (2,2) -- (2,2.5) --  (2.5,3.16) -- (3.16,3.16) -- (3.16,2.5) -- (2.5,2) -- (2,2);
%\draw[thick] (1.5,1.83) -- (3.5,4.5);
%\draw[ultra thick, blue] (2,2.5) -- (2.5,3.16);
%\draw[fill,red] (1.5,1.83) circle [radius=0.05];
%\draw[fill,red] (3.5,4.5) circle [radius=0.05];

\path[fill=yellow] (1.5,1.5) -- (1.5,1.83) -- (3.5,4.5) -- (4.5,4.5)-- (4.5,3.5) -- (1.83,1.5) -- (1.5,1.5);
\draw[thick](1.5,1.5) -- (1.5,1.83) -- (3.5,4.5) -- (4.5,4.5)-- (4.5,3.5) -- (1.83,1.5) -- (1.5,1.5);  
\draw[ultra thick, red] (1.83,1.5) -- (1.5,1.5) -- (1.5,1.83);
\draw[ultra thick, red] (3.5,4.5) -- (4.5,4.5) -- (4.5,3.5);

\draw[->] (0,0.5) -- (5, 0.5);
\draw[->] (0.5,0) -- (0.5,5);

%\draw (1,0.5)--
\node at (0.3,0.2) {$0$};

\draw (1,0.5)--(1,0.6);

\draw (1.5,0.5)-- (1.5,0.7);
\node at (1.5,0.1) {$\frac{1}{2}$};

\draw (2,0.5) -- (2,0.6);
%\node at (2,0) {$\frac{3}{4}$};

\draw (2.5,0.5) -- (2.5,0.7);
\node at (2.5,0.2) {$1$};
%\draw (3,0.5) --

\draw (3,0.5)-- (3,0.6);

\draw (3.5,0.5) -- (3.5,0.7);
\node at (3.5,0.1) {$\frac{3}{2}$};

\draw (4,0.5) -- (4,0.6);

\draw (4.5,0.5) -- (4.5,0.7);
\node at (4.5,0.2) {$2$};

\draw (0.5,1) -- (0.6,1);

\draw (0.5,1.5) -- (0.7,1.5);
\node at (0.3,1.5) {$\frac{1}{2}$};

\draw (0.5,2) -- (0.6,2);

%\draw (0.5,1.83) -- (0.7, 1.83);
%\node at (0,1.83) {$\frac{2}{3}$};

\draw (0.5,2.5) -- (0.7,2.5);
\node at (0.3,2.5) {$1$};

\draw (0.5,3) -- (0.6,3);

\draw (0.5,3.5) -- (0.7,3.5);
\node at (0.3,3.5) {$\frac{3}{2}$};

\draw (0.5,4) -- (0.6,4);

\draw (0.5,4.5) -- (0.7,4.5);
\node at (0.3,4.5) {$2$}; 

\end{tikzpicture}
\end{tabular}
}\fi
\caption{Sections of the tropical column span of $A$ (yellow) and solutions of the minimization problem (blue point $(1,1)$) and the maximization problem (four thick red segments) of Example~\ref{E-1AuAu1-right}. }
\label{F-1AuAu2}
\end{figure}
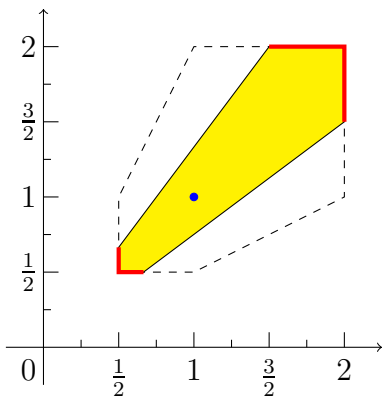

Let us examine problem \eqref{P-maxx1AxAx1} of maximizing the Hilbert seminorm. First, we calculate 
\begin{equation*}
A^{-}
=
\begin{pmatrix}
1 & 4/3 & 2
\\
4/3 & 1 & 2
\\
2 & 2 & 1
\end{pmatrix},
\qquad
\delta
=
1^{T}AA^{-}1
=
2.
\end{equation*}

Let us observe here that since all entries of $A$ are not greater than one, the entrywise logarithm of that matrix is a normal matrix and hence the results of \cite{Delapuente2013Ontropical} apply to it. According to \cite{Delapuente2013Ontropical}, $\delta$ is the greatest entry of $A^{-}$, which is the same as $1^{T}A^{-}1$. However, this is also clear from our computation since $1^{T}A=1^{T}$ and hence $1^{T}AA^{-}1=1^{T}A^{-}1$ in this case.

Next, we have
\begin{equation*}
1^{T}a_{1}
=
1^{T}a_{2}
=
1^{T}a_{3}
=
1,
\end{equation*}
and then conclude that the condition $1^{T}a_{k}a_{lk}^{-1}=\Delta$ is satisfied at four $(k,l)$ pairs: $(3,1)$, $(3,2)$, $(1,3)$ and $(2,3)$.

For $k=3$ and $l=1$, we have to calculate
\begin{equation*}
A_{13}^{-}
=
\begin{pmatrix}
0 & 0 & 0
\\
0 & 0 & 0
\\
2 & 0 & 0
\end{pmatrix},
\quad
(I\oplus A_{13}^{-}A)
=
\begin{pmatrix}
1 & 0 & 0
\\
0 & 1 & 0
\\
2 & 3/2 & 1
\end{pmatrix}.
\end{equation*}
and then we consider the following solutions of \eqref{P-maxx1AxAx1}:
\begin{equation*}
x
=
\begin{pmatrix}
1 & 0 & 0
\\
0 & 1 & 0
\\
2 & 3/2 & 1
\end{pmatrix}u,
\quad
u>0.
\end{equation*}

The corresponding vector in $\operatorname{span}(A)$ is calculated as
\begin{equation*}
Ax
=
\begin{pmatrix}
1 & 3/4 & 1/2
\\
1 & 1 & 1/2
\\
2 & 3/2 & 1
\end{pmatrix}u.
\end{equation*}

The first and third column of the last matrix are proportional to each other, hence solution subset corresponding to $(k,l)=(3,1)$ is 
\begin{equation*}
\left\{
\begin{pmatrix}
1/2 & 1/2
\\
1/2 & 2/3
\\
1 & 1
\end{pmatrix}v
\colon
v>0\right\}.
\end{equation*}

The intersection of this subset with $\{x\colon x_3=1\}$ is the segment with ends $(1/2, 1/2)$ and $(1/2, 2/3)$.

Assume that $(k,l)=(3,2)$. In this case, we have
\begin{equation*}
A_{23}^{-}
=
\begin{pmatrix}
0 & 0 & 0
\\
0 & 0 & 0
\\
0 & 2 & 0
\end{pmatrix},
\quad
(I\oplus A_{23}^{-}A)
=
\begin{pmatrix}
1 & 0 & 0
\\
0 & 1 & 0
\\
3/2 & 2 & 1
\end{pmatrix}.
\end{equation*}

The solution of problem \eqref{P-maxx1AxAx1} is written in the form
\begin{equation*}
x
=
\begin{pmatrix}
1 & 0 & 0
\\
0 & 1 & 0
\\
2 & 3/2 & 1
\end{pmatrix}u,
\quad
u>0.
\end{equation*}

The solution yields the vector
\begin{equation*}
Ax
=
\begin{pmatrix}
1 & 1 & 1/2
\\
3/4 & 1 & 1/2
\\
3/2 & 2 & 1
\end{pmatrix}u.
\end{equation*}

The second and third column of the matrix are proportional to each other, hence solution subset corresponding to $(k,l)=(3,2)$ is 
\begin{equation*}
\left\{
\begin{pmatrix}
1/2 & 2/3
\\
1/2 & 1/2
\\
1 & 1
\end{pmatrix}v
\colon
v>0
\right\}.
\end{equation*}

The intersection of this subset with the plane $\{x\colon x_{3}=1\}$ is the segment with the ends $(1/2, 1/2)$ and $(2/3, 1/2)$.

We can similarly find solutions for $(k,l)=(1,3)$ and $(k,l)=(2,3)$. Solutions for all four pairs $(k,l)$ are shown on Figure~\ref{F-1AuAu2} (thick red lines).
\end{example}

\section{Tropical implementation of AHP}
\label{s:AHP}

We are now in a position to describe our tropical implementation of AHP. Let us consider a multi-criteria decision problem to rate $n$ alternatives (choices) from pairwise comparisons with respect to $m$ criteria. Suppose that $C$ is an $(m\times m)$-matrix of pairwise comparisons of the criteria, and $A_{1},\ldots,A_{m}$ are $(n\times n)$-matrices of pairwise comparisons of the alternatives for every criterion. Given the above matrices, the problem consists in finding $n$-vectors $x$ of scores (rates, priorities) of the alternatives.   

We propose a decision procedure that involves the following steps: (i) log-Chebyshev approximation of the pairwise comparison matrix of criteria to find the vector of weights in a parametric form; (ii) simultaneous weighted minimax approximation of the pairwise comparison matrices of choices according to each criterion to obtain the vectors of priorities for choices; (iii) solution of the optimization problems of maximizing and minimizing the Hilbert seminorm to derive the vectors, which most and least differentiate between the choices with the highest and lowest priorities. 

At the first step, we need to evaluate the relative importance of the criteria by solving problem \eqref{P-minxaijxjxi} with the matrix $C$. In terms of the tropical linear algebra, problem \eqref{P-minxaijxjxi} takes the form of \eqref{P-minxxAx}. Therefore, we can apply Lemma~\ref{L-minxxAx} to obtain the vector of weights in the parametric form
\begin{equation*}
w
=
(\lambda^{-1}C)^{\ast}v,
\end{equation*}
where $\lambda$ is the tropical spectral radius of $C$, and $v$ is any positive vector.

The next step is the evaluation of the priorities of alternatives, which involves the solution of problem \eqref{P-minxaijxjxi} with the matrix $B$ defined by \eqref{E-bij-maxwkaijk}, where the vector $w$ is the weight vector obtained at the first step. After translation into the tropical linear algebra language, we have problem \eqref{P-minxwkxAkx}. Corollary~\ref{C-minxwkxAkx} offers the solution to the problem in the form
\begin{equation}
x
=
(\mu^{-1}B)^{\ast}u,
\qquad
B
=
\bigoplus_{k=1}^{m}w_{k}A_{k},
\label{e:xB}
\end{equation}
where
$\mu$ is the spectral radius of the matrix $B$, and $u$ is any positive vector.

If the obtained solution $x=Su$, where $S=(\mu^{-1}B)^{\ast}$ or $S$ is a submatrix of linearly independent columns of $(\mu^{-1}B)^{\ast}$, is not unique (up to a positive factor), we need to solve problems \eqref{P-maxxmaxximax1xj} and \eqref{P-minxmaxximax1xj} to determine the most and least differentiating priority vectors.

In the tropical linear algebra setting, problem \eqref{P-maxxmaxximax1xj} reduces to \eqref{P-maxx1AxAx1} where $A$ is replaced by $S$, and $x$ by $u$. Application of Corollary~\ref{C-maxx1AxAx1} to solve the last problem requires calculating $\Delta=1^{T}SS^{-}1$, and yields the solution 
\begin{equation*}
u
=
S(I\oplus S_{lk}^{-}S)v,
\qquad
v>0,
\end{equation*}
where the indices $k$ and $l$ satisfy the condition
\begin{equation*}
1^{T}s_{k}s_{lk}^{-1}
=
\Delta.
\end{equation*}

The most differentiating priority vectors are then given by
\begin{equation*}
x_{1}
=
S(I\oplus S_{lk}^{-}S)v,
\qquad
v>0.
\end{equation*}

Problem \eqref{P-minxmaxximax1xj} takes the form of \eqref{P-minx1xx1-xlambdaAastu} with $A$ replaced by $B$ and $\lambda$ by $\mu$. The solution is given by Corollary~\ref{C-minx1xx1-xlambdaAastu}, and involves calculating $\delta=1^{T}(\mu^{-1}B)^{\ast}1$, which is used to obtain the least differentiating vector of priorities
\begin{equation*}
x_{2}
=
(\delta^{-1}11^{T}\oplus\mu^{-1}B)^{\ast}u,
\qquad
u>0.
\end{equation*}

We now present two examples, which will illustrate the computational technique involved in the tropical implementation of AHP described above. In the first examples, the matrix $C$ has an essentially unique weight vector $w$ associated with it, which is then used to form the matrix $B$ as in \eqref{e:xB}.

In the second example, the weight vector associated with $C$ is not unique and represented in a parametric form. In this case, we combine the choice of the appropriate weight vector with the solution of the optimization problems to find the most and least differentiating vectors on the next step of the procedure.

\subsection{Vacation site selection example}
\label{S-VSSE}

Consider an example from \cite{Saaty1977Scaling}, where a plan for vacation is to be selected. The places considered are $\mathbf{S}$: short trips from Philadelphia (i.e., New York, Washington, Atlantic City, New Hope, etc.), $\mathbf{Q}$: Quebec, $\mathbf{D}$: Denver, $\mathbf{C}$: California. The problem is to evaluate the places with respect to the following criteria: (1) cost of the trip from Philadelphia, (2) sight-seeing opportunities, (3) entertainment (doing things), (4) way of travel, (5) eating places. 

The comparison matrix of criteria for places is given by
\begin{equation*}
C
=
\begin{pmatrix}
1 & 1/5 & 1/5 & 1 & 1/3
\\
5 & 1 & 1/5 & 1/5 & 1
\\
5 & 5 & 1 & 1/5 & 1
\\
1 & 5 & 5 & 1 & 5
\\
3 & 1 & 1 & 1/5 & 1
\end{pmatrix}.
\end{equation*}

The pairwise comparison matrices of vacation sites with respect to the criteria are defined as follows:
\begin{gather*}
A_{1}
=
\begin{pmatrix}
1 & 3 & 7 & 9
\\
1/3 & 1 & 6 & 7
\\
1/7 & 1/6 & 1 & 3
\\
1/9 & 1/7 & 1/3 & 1
\end{pmatrix},
\qquad
A_{2}
=
\begin{pmatrix}
1 & 1/5 & 1/6 & 1/4
\\
5 & 1 & 2 & 4
\\
6 & 1/2 & 1 & 6
\\
4 & 1/4 & 1/6 & 1
\end{pmatrix},
\\
A_{3}
=
\begin{pmatrix}
1 & 7 & 7 & 1/2
\\
1/7 & 1 & 1 & 1/7
\\
1/7 & 1 & 1 & 1/7
\\
2 & 7 & 7 & 1
\end{pmatrix},
\qquad
A_{4}
=
\begin{pmatrix}
1 & 4 & 1/4 & 1/3
\\
1/4 & 1 & 1/2 & 3
\\
4 & 2 & 1 & 3
\\
3 & 1/3 & 1/3 & 1
\end{pmatrix},
\\
A_{5}
=
\begin{pmatrix}
1 & 1 & 7 & 4
\\
1 & 1 & 6 & 3
\\
1/7 & 1/6 & 1 & 1/4
\\
1/4 & 1/3 & 4 & 1
\end{pmatrix}.
\end{gather*}

To solve the problem, we first evaluate the priorities of criteria. We take the pairwise comparison matrix $C$, and find its tropical spectral radius (its maximum cycle geometric mean). Using \eqref{E-lambda-ai1i2ai2i3aiki1}, we obtain
\begin{equation*}
\lambda
=
(c_{14}c_{43}c_{32}c_{21})^{1/4}
=
5^{3/4}
\approx
3.3437.
\end{equation*}

Furthermore, we consider the matrix
\begin{equation*}
\lambda^{-1}C
=
\begin{pmatrix}
1/\lambda & 1/5\lambda & 1/5\lambda & 1/\lambda & 1/3\lambda
\\
5/\lambda & 1/\lambda & 1/5\lambda & 1/5\lambda & 1/\lambda
\\
5/\lambda & 5/\lambda & 1/\lambda & 1/5\lambda & 1/\lambda
\\
1/\lambda & 5/\lambda & 5/\lambda & 1/\lambda & 5/\lambda
\\
3/\lambda & 1/\lambda & 1/\lambda & 1/5\lambda & 1/\lambda
\end{pmatrix},
\end{equation*}
and calculate its powers to obtain the Kleene star matrix 
\begin{multline*}
(\lambda^{-1}C)^{\ast}
=
I\oplus\lambda^{-1}C\oplus\lambda^{-2}C^{2}\oplus\lambda^{-3}C^{3}\oplus\lambda^{-4}C^{4}
\\
=
\begin{pmatrix}
1 & \lambda/5 & 5/\lambda^{2} & 1/\lambda & 5/\lambda^{2}
\\
5/\lambda & 1 & \lambda/5 & 5/\lambda^{2} & \lambda/5
\\
\lambda^{2}/5 & 5/\lambda & 1 & \lambda/5 & 1
\\
\lambda & \lambda^{2}/5 & 5/\lambda & 1 & 5/\lambda
\\
3/\lambda & 3/5 & 3\lambda/25 & 3/\lambda^{2} & 3\lambda/25
\end{pmatrix}.
\end{multline*}

The columns of the Kleene matrix generate the set of all weight vectors of criteria. Since all columns of this matrix are collinear, any one of them can serve as the weight vector. We take the first column, and use its elements as coefficients to combine the matrices $A_{1},\ldots,A_{5}$ into one matrix
\begin{equation*}
B
=
A_{1}
\oplus
5\lambda^{-1}
A_{2}
\oplus
%25\lambda^{-2}
5^{-1}\lambda^{2}
A_{3}
\oplus
\lambda
A_{4}
\oplus
3\lambda^{-1}
A_{5}
\\
=
\begin{pmatrix}
\lambda & 7\lambda^{2}/5 & 7\lambda^{2}/5 & 9
\\
25/\lambda & \lambda & 6 & 3\lambda
\\
4\lambda & 2\lambda & \lambda & 3\lambda
\\
3\lambda & 7\lambda^{2}/5 & 7\lambda^{2}/5 & \lambda
\end{pmatrix}.
\end{equation*}

We now apply Corollary~\ref{C-minxwkxAkx} to find all priority vectors that correspond to the matrix $B$. Evaluation of the tropical spectral radius (the maximum cycle mean) of $B$ yields
\begin{equation*}
\mu
=
(b_{13}b_{31})^{1/2}
%=
%140^{1/2}5^{1/8}
=
2\cdot5\cdot7^{1/2}/\lambda^{1/2}
=
2\cdot5^{5/8}7^{1/2}
\approx
14.4689.
\end{equation*}

Furthermore, we calculate powers of the matrix
\begin{equation*}
\mu^{-1}B
=
\begin{pmatrix}
\lambda/\mu & 7\lambda^{2}/5\mu & 7\lambda^{2}/5\mu & 9/\mu
\\
25/\lambda\mu & \lambda/\mu & 6/\mu & 3\lambda/\mu
\\
4\lambda/\mu & 2\lambda/\mu & \lambda/\mu & 3\lambda/\mu
\\
3\lambda/\mu & 7\lambda^{2}/5\mu & 7\lambda^{2}/5\mu & \lambda/\mu
\end{pmatrix},
\end{equation*}
and combine them to construct the matrix
\begin{equation*}
(\mu^{-1}B)^{\ast}
=
I\oplus\mu^{-1}B\oplus\mu^{-2}B^{2}\oplus\mu^{-3}B^{3}
=
\begin{pmatrix}
1 & \mu/4\lambda & \mu/4\lambda & 3/4
\\
3\lambda/\mu & 1 & 3/4 & 3\lambda/\mu
\\
4\lambda/\mu & 1 & 1 & 3\lambda/\mu
\\
1 & \mu/4\lambda & \mu/4\lambda & 1
\end{pmatrix}
\end{equation*}
whose columns generate all priority vectors for alternatives.

Observing that the first column is collinear with the third, one of them, say the third column, can be removed from the set of generators. Thus, we represent a complete solution as the set of vectors
\begin{equation*}
x
=
Su,
\quad
S
=
\begin{pmatrix}
1 & \mu/4\lambda & 3/4
\\
3\lambda/\mu & 1 & 3\lambda/\mu
\\
4\lambda/\mu & 1 & 3\lambda/\mu
\\
1 & \mu/4\lambda & 1
\end{pmatrix},
\quad
u>0.
\end{equation*}

To find solutions that most differentiate alternatives with the highest and lowest priorities, we apply Corollary~\ref{C-maxx1AxAx1}. We start with the calculation
\begin{align*}
1^{T}s_{1}
&=
1,
&
1^{T}s_{2}
&=
\mu/4\lambda,
&
1^{T}s_{3}
&=
1,
\\
s_{1}^{-}1
&=
\mu/3\lambda,
&
s_{2}^{-}1
&=
1,
&
s_{3}^{-}1
&=
\mu/3\lambda,
\end{align*}
and then obtain
\begin{equation*}
\Delta
%=
%1^{T}SS^{-}1
=
1^{T}s_{1}s_{1}^{-}1
\oplus
1^{T}s_{2}s_{2}^{-}1
\oplus
1^{T}s_{3}s_{3}^{-}1
=
\mu/3\lambda
=
2\cdot3^{-1}5^{-1/8}7^{1/2}
\approx
1.4424.
\end{equation*}

The condition $1^{T}s_{k}s_{lk}^{-1}=\Delta$ holds if we take the following $(k,l)$ pairs: $(1,2)$, $(3,2)$ and $(3,3)$. 

First, assume that $k=1$ and $l=2$. We form the matrices
\begin{equation*}
S_{21}
=
\begin{pmatrix}
0 & 0 & 0
\\
3\lambda/\mu & 0 & 0
\\
0 & 0 & 0
\\
0 & 0 & 0
\end{pmatrix},
\quad
S_{21}^{-}S
=
\begin{pmatrix}
1 & \mu/3\lambda & 1
\\
0 & 0 & 0
\\
0 & 0 & 0
\end{pmatrix},
\end{equation*}
and then derive the matrix, which generates the most differentiating priority vectors,
\begin{equation*}
S(I\oplus S_{21}^{-}S)
=
\begin{pmatrix}
1 & \mu/3\lambda & 1
\\
3\lambda/\mu & 1 & 3\lambda/\mu
\\
4\lambda/\mu & 4/3 & 4\lambda/\mu
\\
1 & \mu/3\lambda & 1
\end{pmatrix}.
\end{equation*}

Since all columns in the last matrix are collinear to each other, we take one of them, say the first, to write one of the most differentiating solutions as
\begin{equation*}
x_{1}^{\prime}
=
\begin{pmatrix}
1
\\
3\lambda/\mu
\\
4\lambda/\mu
\\
1
\end{pmatrix}
u,
\quad
\lambda
=
5^{3/4},
\quad
\mu
=
2\cdot5^{5/8}7^{1/2},
\quad
u>0.
\end{equation*}

Specifically, by setting $u=1$, we have $x_{1}^{\prime}\approx(1.0000, 0.6933, 0.9244, 1.0000)^{T}$. This vector specifies the order of choices as $\mathbf{C}\equiv\mathbf{S}\succ\mathbf{D}\succ\mathbf{Q}$.

Next, we examine the case where $k=3$ and $l=2$. In a similar way, we obtain the vector 
\begin{equation*}
x_{1}^{\prime\prime}
=
\begin{pmatrix}
3/4
\\
3\lambda/\mu
\\
3\lambda/\mu
\\
1
\end{pmatrix}
u,
\quad
\lambda
=
5^{3/4},
\quad
\mu
=
2\cdot5^{5/8}7^{1/2},
\quad
u>0,
\end{equation*}
which suggests another most differentiating solution.

If $u=1$, then $x_{1}^{\prime\prime}\approx(0.7500, 0.6933, 0.6933, 1.0000)^{T}$, which puts the choices in the order $\mathbf{C}\succ\mathbf{S}\succ\mathbf{D}\equiv\mathbf{Q}$.

It is not difficult to verify that the case with $k=3$ and $l=3$ introduces no other solutions than those already found.

We now turn to an application of Corollary~\ref{C-minx1xx1-xlambdaAastu} to derive the least differentiating vector of priorities; as it will turn out, in this example it is essentially unique. First, we calculate
\begin{equation*}
\delta
=
1^{T}(\mu^{-1}B)^{\ast}1
=
\mu/4\lambda
=
2^{-1}5^{-1/8}7^{1/2}
\approx
1.0818,
\end{equation*}
and then construct the matrix
\begin{equation*}
\delta^{-1}11^{T}
\oplus
\mu^{-1}B
=
\begin{pmatrix}
1/\delta & \delta & \delta & 1/\delta 
\\
1/\delta & 1/\delta & 1/\delta & 1/\delta
\\
1/\delta & 1/\delta & 1/\delta & 1/\delta
\\
1/\delta & \delta & \delta & 1/\delta
\end{pmatrix}.
\end{equation*}

The least differentiating priority vectors are generated by the columns of the Kleene star matrix
\begin{multline*}
(\delta^{-1}11^{T}\oplus\mu^{-1}B)^{\ast}
\\
=
I
\oplus
(\delta^{-1}11^{T}\oplus\mu^{-1}B)
\oplus
(\delta^{-1}11^{T}\oplus\mu^{-1}B)^{2}
\oplus
(\delta^{-1}11^{T}\oplus\mu^{-1}B)^{3}
\\
=
\begin{pmatrix}
1 & \delta & \delta & 1 
\\
1/\delta & 1 & 1 & 1/\delta
\\
1/\delta & 1 & 1 & 1/\delta
\\
1 & \delta & \delta & 1 
\end{pmatrix}.
\end{multline*}

Observing that all columns in the matrix obtained are collinear, we take one of them, say the first, to write the least differentiating solutions as
\begin{equation*}
x_{2}
=
\begin{pmatrix}
1
\\
1/\delta
\\
1/\delta
\\
1
\end{pmatrix}u,
\quad
u>0.
\end{equation*}

Setting $u=1$, we have $x_{2}\approx(1, 0.9244, 0.9244, 1)^{T}$. This vector arranges the alternatives in the order $\mathbf{C}\equiv\mathbf{S}\succ\mathbf{D}\equiv\mathbf{Q}$.

As one can see, all solutions indicate the highest score of the fourth choice (California). The score assigned to the first choice (short trip) is the same or lower. The third choice (Denver) has the same or higher score, than the second choice (Quebec), and both of them always have a lower score than the first. Combining both the most and least differentiating solutions yields the order of choices $\mathbf{C}\succeq\mathbf{S}\succ\mathbf{D}\succeq\mathbf{Q}$.

Note that the results obtained above with the tropical implementation of AHP are quite different from those offered by the classical AHP method. Specifically, the order of choices, found in \cite{Saaty1977Scaling}, is $\mathbf{S}\succ\mathbf{D}\succ\mathbf{C}\succ\mathbf{Q}$.

\subsection{School selection example}
\label{S-SSE}

As another example, we investigate a problem in \cite{Saaty1977Scaling,Saaty1990Analytic} to rank three high schools $\mathbf{A}$, $\mathbf{B}$ and $\mathbf{C}$, according to the following characteristics (criteria): (1) learning, (2) friends, (3) school life, (4) vocational training, (5) college preparation, (6) music classes.

The results of pairwise comparison of criteria are given by the matrix
\begin{equation*}
C
=
\begin{pmatrix}
1 & 4 & 3 & 1 & 3 & 4
\\
1/4 & 1 & 7 & 3 & 1/5 & 1
\\
1/3 & 1/7 & 1 & 1/5 & 1/5 & 1/6
\\
1 & 1/3 & 5 & 1 & 1 & 1/3
\\
1/3 & 5 & 5 & 1 & 1 & 3
\\
1/4 & 1 & 6 & 3 & 1/3 & 1
\end{pmatrix}.
\end{equation*}

The matrices of pairwise comparison of schools for each criterion take the following forms:
\begin{gather*}
A_{1}
=
\begin{pmatrix}
1 & 1/3 & 1/2
\\
3 & 1 & 3
\\
2 & 1/3 & 1
\end{pmatrix},
\quad
A_{2}
=
\begin{pmatrix}
1 & 1 & 1
\\
1 & 1 & 1
\\
1 & 1 & 1
\end{pmatrix},
\quad
A_{3}
=
\begin{pmatrix}
1 & 5 & 1
\\
1/5 & 1 & 1/5
\\
1 & 5 & 1
\end{pmatrix},
\\
A_{4}
=
\begin{pmatrix}
1 & 9 & 7
\\
1/9 & 1 & 1/5
\\
1/7 & 5 & 1
\end{pmatrix},
\quad
A_{5}
=
\begin{pmatrix}
1 & 1/2 & 1
\\
2 & 1 & 2
\\
1 & 1/2 & 1
\end{pmatrix},
\quad
A_{6}
=
\begin{pmatrix}
1 & 6 & 4
\\
1/6 & 1 & 1/3
\\
1/4 & 3 & 1
\end{pmatrix}.
\end{gather*}

The solution of the problem involves evaluation of the priority vectors that most and least differentiate the schools with the highest and lowest priorities. To find the most differentiating solution, we first obtain a parametric description of all weight vectors from the pairwise comparison matrix of criteria. The components of the weight vector are used to form a weighted sum of comparison matrices of schools for each criteria. Next, the priority vectors for schools are evaluated based on the sum of matrices with parameterized weights. Finally, those priority vectors, which most and least differentiate the schools with the highest and lowest scores, are taken as the most and leats differentiating solutions to the problem.

As the first step, we need to obtain the priority vectors for criteria, which specify the weights of the criteria. We evaluate the spectral radius of the matrix $C$ by using \eqref{E-lambda-ai1i2ai2i3aiki1} to write 
\begin{equation*}
\lambda
=
(c_{15}c_{52}c_{24}c_{41})^{1/4}
%=
%(45)^{1/4}
=
3^{1/2}5^{1/4}
\approx
2.5900.
\end{equation*}

We calculate the first five powers of the matrix
\begin{equation*}
\lambda^{-1}C
=
\begin{pmatrix}
1/\lambda & 4/\lambda & 3/\lambda & 1/\lambda & 3/\lambda & 4/\lambda
\\
1/4\lambda & 1/\lambda & 7/\lambda & 3/\lambda & 1/5\lambda & 1/\lambda
\\
1/3\lambda & 1/7\lambda & 1/\lambda & 1/5\lambda & 1/5\lambda & 1/6\lambda
\\
1/\lambda & 1/3\lambda & 5/\lambda & 1/\lambda & 1/\lambda & 1/3\lambda
\\
1/3\lambda & 5/\lambda & 5/\lambda & 1/\lambda & 1/\lambda & 3/\lambda
\\
1/4\lambda & 1/\lambda & 6/\lambda & 3/\lambda & 1/3\lambda & 1/\lambda
\end{pmatrix},
\end{equation*}
and then combine these powers to derive the Kleene star matrix
\begin{multline*}
(\lambda^{-1}C)^{\ast}
=
I\oplus\lambda^{-1}C\oplus\lambda^{-2}C^{2}\oplus\lambda^{-3}C^{3}\oplus\lambda^{-4}C^{4}\oplus\lambda^{-5}C^{5}
\\
=
\begin{pmatrix}
1 & \lambda^{2}/3 & 7\lambda/3 & \lambda & 3/\lambda & 4/\lambda
\\
3/\lambda^{2} & 1 & 7/\lambda & 3/\lambda & \lambda/5 & 4\lambda/15
\\
1/3\lambda & \lambda/9 & 1 & 1/3 & 1/\lambda^{2} & 4/3\lambda^{2}
\\
1/\lambda & \lambda/3 & 7/3 & 1 & 3/\lambda^{2} & 4/\lambda^{2}
\\
\lambda/3 & 5/\lambda & 7\lambda^{2}/9 & \lambda^{2}/3 & 1 & 4/3
\\
3/\lambda^{2} & 1 & 7/\lambda & 3/\lambda & \lambda/5 & 1
\end{pmatrix}.
\end{multline*}

Note that the first, second, fourth and fifth columns in the matrix $(\lambda^{-1}C)^{\ast}$ are collinear, and thus all of them but one, say the fourth, can be omitted. We combine the fourth column together with the third multiplied by $3/7$ and the sixth multiplied by $\lambda^{2}/4$ to obtain the generating matrix, and introduce the vector of parameters $v=(v_{1},v_{2},v_{3})^{T}>0$ to represent the weight vector in parametric form as 
\begin{equation*}
w
=
\begin{pmatrix}
\lambda & \lambda & \lambda
\\
3/\lambda & 3/\lambda & 3/\lambda
\\
3/7 & 1/3 & 1/3
\\
1 & 1 & 1
\\
\lambda^{2}/3 & \lambda^{2}/3 & \lambda^{2}/3
\\
3/\lambda & 3/\lambda & \lambda^{2}/4
\end{pmatrix}
v,
\quad
v>0.
\end{equation*}

We now use the components of the vector $w$ as weights to combine the matrices $A_{1},\ldots,A_{6}$ into the matrix
\begin{multline*} 
B
=
\lambda(v_{1}\oplus v_{2}\oplus v_{3})
A_{1}
\oplus
3\lambda^{-1}(v_{1}\oplus v_{2}\oplus v_{3})
A_{2}
\oplus
(3\cdot7^{-1}v_{1}\oplus3^{-1}v_{2}\oplus3^{-1}v_{3})
A_{3}
\\
\oplus
(v_{1}\oplus v_{2}\oplus v_{3})
A_{4}
\oplus
3^{-1}\lambda^{2}(v_{1}\oplus v_{2}\oplus v_{3})
A_{5}
\oplus
(3\lambda^{-1}v_{1}
\oplus
3\lambda^{-1}v_{2}
\oplus
4^{-1}\lambda^{2}v_{3})
A_{6}
\\
=
\begin{pmatrix}
\lambda(v_{1}\oplus v_{2}\oplus v_{3}) & 9(v_{1}\oplus v_{2})\oplus3\lambda^{2}v_{3}/2 & 7(v_{1}\oplus v_{2}\oplus v_{3})
\\
3\lambda(v_{1}\oplus v_{2}\oplus v_{3}) & \lambda(v_{1}\oplus v_{2}\oplus v_{3}) & 3\lambda(v_{1}\oplus v_{2}\oplus v_{3})
\\
2\lambda(v_{1}\oplus v_{2}\oplus v_{3}) & 5(v_{1}\oplus v_{2})\oplus3\lambda^{2}v_{3}/4 & \lambda(v_{1}\oplus v_{2}\oplus v_{3})
\end{pmatrix}.
\end{multline*} 

Observing that the parameters $v_{1}$ and $v_{2}$ occur in all entries of the matrix in the form of the sum $v_{1}\oplus v_{2}$, we change the variables by replacing this sum by $v_{1}$ and $v_{3}$ by $v_{2}$ to rewrite the matrix in the more simple form
\begin{equation*} 
B
=
\begin{pmatrix}
\lambda(v_{1}\oplus v_{2}) & 9v_{1}\oplus3\lambda^{2}v_{2}/2 & 7(v_{1}\oplus v_{2})
\\
3\lambda(v_{1}\oplus v_{2}) & \lambda(v_{1}\oplus v_{2}) & 3\lambda(v_{1}\oplus v_{2})
\\
2\lambda(v_{1}\oplus v_{2}) & 5v_{1}\oplus3\lambda^{2}v_{2}/4 & \lambda(v_{1}\oplus v_{2})
\end{pmatrix}.
\end{equation*} 

Furthermore, we take the matrix $B$ to derive all solutions by using Corollary~\ref{C-minxwkxAkx}. Evaluation of the spectral radius of the matrix $B$ yields
\begin{equation*}
\mu
=
(b_{12}b_{21})^{1/2}
=
(3\lambda(v_{1}\oplus v_{2})(9v_{1}\oplus3\lambda^{2}v_{2}/2))^{1/2}.
\end{equation*}

We consider the matrix $\mu^{-1}B$ and calculate the Kleene star matrix
\begin{multline*}
(\mu^{-1}B)^{\ast}
=
I\oplus\mu^{-1}B\oplus\mu^{-2}B^{2}
\\
=
\begin{pmatrix}
1 & \mu/3\lambda(v_{1}\oplus v_{2}) & 1
\\
3\lambda(v_{1}\oplus v_{2})/\mu & 1 & 3\lambda(v_{1}\oplus v_{2})/\mu
\\
2\lambda(v_{1}\oplus v_{2})/\mu & 2/3 & 1
\end{pmatrix}.
\end{multline*}

As the first two columns of the obtained matrix are collinear, we take one of them, say the second, to write the solution in the form
\begin{equation*}
x
=
Su,
\quad
S
=
\begin{pmatrix}
\mu/3\lambda(v_{1}\oplus v_{2}) & 1
\\
1 & 3\lambda(v_{1}\oplus v_{2})/\mu
\\
2/3 & 1
\end{pmatrix},
\quad
u>0.
\end{equation*}

Since the solution is not unique up to a positive factor, we need to find the vectors, which most and least differentiate between the alternatives with the highest and lowest priorities. We begin with the application of Corollary~\ref{C-maxx1AxAx1} to obtain the most differentiating solutions of the problem. First, we note that $\mu=(3\lambda(v_{1}\oplus v_{2})(9v_{1}\oplus3\lambda^{2}v_{2}/2))^{1/2}>3\lambda(v_{1}\oplus v_{2})$, and calculate
\begin{align*}
1^{T}s_{1}
&=
\mu/3\lambda(v_{1}\oplus v_{2}),
&
1^{T}s_{2}
&=
1,
\\
s_{1}^{-}1
&=
3/2,
&
s_{2}^{-}1
&=
\mu/3\lambda(v_{1}\oplus v_{2}).
\end{align*}

Next, we have to find vectors $v$ that maximize
\begin{equation*}
\Delta_{v}
%=
%1^{T}SS^{-}1
=
1^{T}s_{1}s_{1}^{-}1
\oplus
1^{T}s_{2}s_{2}^{-}1
=
\mu/2\lambda(v_{1}\oplus v_{2})
=
\left(\frac{3(9v_{1}\oplus3\lambda^{2}v_{2}/2)}{4\lambda(v_{1}\oplus v_{2})}\right)^{1/2}.
\end{equation*}

Observing that $9<3\lambda^{2}/2$, we see that the maximum of $\Delta_{v}$ is attained if and only if $v_{2}\geq v_{1}$, and equal to $\Delta=(9\lambda/8)^{1/2}$. In this case, we have
\begin{equation*}
\mu
=
(9\lambda^{3}/2))^{1/2}v_{2},
\end{equation*}
whereas the matrix $S$ becomes
\begin{equation*}
S
=
\begin{pmatrix}
(\lambda/2)^{1/2} & 1
\\
1 & (2/\lambda)^{1/2}
\\
2/3 & 1
\end{pmatrix}.
\end{equation*}

The condition $1^{T}s_{k}s_{lk}^{-1}=\Delta$ holds if we take $k=1$ and $l=3$. According to Corollary~\ref{C-maxx1AxAx1}, we construct the matrices
\begin{equation*}
S_{31}
=
\begin{pmatrix}
0 & 0
\\
0 & 0
\\
2/3 & 0
\end{pmatrix},
\qquad
S_{31}^{-}S
=
\begin{pmatrix}
1 & 3/2
\\
0 & 0
\end{pmatrix},
\end{equation*}
and then calculate the generating matrix
\begin{equation*}
S(I\oplus S_{31}^{-}S)
=
\begin{pmatrix}
(\lambda/2)^{1/2} & 3(\lambda/2)^{1/2}/2
\\
1 & 3/2
\\
2/3 & 1
\end{pmatrix}.
\end{equation*}

Since both columns of the generating matrix are collinear, we take the first column to write the solution
\begin{equation*}
x_{1}
=
\begin{pmatrix}
(\lambda/2)^{1/2}
\\
1
\\
2/3
\end{pmatrix}
u,
\quad
\lambda
=
3^{1/2}5^{1/4}
\approx
2.5900,
\quad
u>0.
\end{equation*}

With $u=(\lambda/2)^{-1/2}$, we have the vector of ratings $x_{1}\approx(1,0.8787,0.5858)^{T}$, which gives the order $\mathbf{A}\succ\mathbf{B}\succ\mathbf{C}$.

Let us derive the least differentiating solution by using Corollary~\ref{C-minx1xx1-xlambdaAastu}. We take the matrix $\mu^{-1}B$ and calculate
\begin{equation*}
\delta_{v}
=
1^{T}(\mu^{-1}B)^{\ast}1
=
\mu/3\lambda(v_{1}\oplus v_{2})
=
\left(\dfrac{9v_{1}\oplus3\lambda^{2}v_{2}/2}{3\lambda(v_{1}\oplus v_{2})}\right)^{1/2}.
\end{equation*}

To find the minimum of $\delta_{v}$ with respect to $v$, consider two cases. First, assume that $v_{1}\leq v_{2}$. Observing that $3\lambda^{2}/2>9$, we have
\begin{equation*} 
\delta_{v}
=
2^{-1/2}\lambda^{1/2}
\approx
1.1380.
\end{equation*}

If $v_{1}>v_{2}$, we obtain the lower bound
\begin{equation*} 
\delta_{v}
=
(3/\lambda
\oplus
\lambda v_{2}/2v_{1})^{1/2}
\geq
3^{1/2}\lambda^{-1/2}
=
3^{1/4}5^{-1/8}
\approx
1.0762.
\end{equation*}

This bound is achieved if $v_{2}\leq6v_{1}/\lambda^{2}$, and thus is the minimum $\delta$ under consideration.

Finally, under the condition $v_{2}\leq6v_{1}/\lambda^{2}<v_{1}$, we have
\begin{equation*} 
B
=
v_{1}
\begin{pmatrix}
\lambda & 9 & 7
\\
3\lambda & \lambda & 3\lambda
\\
2\lambda & 5 & \lambda
\end{pmatrix},
\quad
\mu
=
3^{3/2}\lambda^{1/2}v_{1},
\quad
\delta
=
3^{1/2}\lambda^{-1/2}
=
9v_{1}/\mu.
\end{equation*} 

The matrix $\mu^{-1}B$ takes the form
\begin{equation*} 
\mu^{-1}B
=
\begin{pmatrix}
\lambda^{1/2}/3^{3/2} & 3^{1/2}/\lambda^{1/2} & 7/3^{3/2}\lambda^{1/2}
\\
\lambda^{1/2}/3^{1/2} & \lambda^{1/2}/3^{3/2} & \lambda^{1/2}/3^{1/2}
\\
2\lambda^{1/2}/3^{3/2} & 5/3^{3/2}\lambda^{1/2} & \lambda^{1/2}/3^{3/2}
\end{pmatrix}
=
\begin{pmatrix}
1/3\delta & \delta & 7\delta/9
\\
1/\delta & 1/3\delta & 1/\delta
\\
2/3\delta & 5\delta/9 & 1/3\delta
\end{pmatrix}.
\end{equation*}

Furthermore, we construct the matrix
\begin{equation*}
\delta^{-1}11^{T}
\oplus
\mu^{-1}B
=
\delta^{-1}11^{T}
\oplus
\begin{pmatrix}
1/3\delta & \delta & 7\delta/9
\\
1/\delta & 1/3\delta & 1/\delta
\\
2/3\delta & 5\delta/9 & 1/3\delta
\end{pmatrix}
=
\begin{pmatrix}
1/\delta & \delta & 1/\delta
\\
1/\delta & 1/\delta & 1/\delta
\\
1/\delta & 1/\delta & 1/\delta
\end{pmatrix},
\end{equation*}
and then find the matrix
\begin{multline*}
(\delta^{-1}11^{T}\oplus\mu^{-1}B)^{\ast}
=
I
\oplus
(\delta^{-1}11^{T}\oplus\mu^{-1}B)
\oplus
(\delta^{-1}11^{T}\oplus\mu^{-1}B)^{2}
\\
=
\begin{pmatrix}
1 & \delta & 1
\\
1/\delta & 1 & 1/\delta
\\
1/\delta & 1 & 1
\end{pmatrix}.
\end{multline*}

Since the first two columns in the matrix are collinear, we take the first one to write a solution, which least differentiate the alternatives, as
\begin{equation*}
x_{2}^{\prime}
=
\begin{pmatrix}
1
\\
1/\delta
\\
1/\delta
\end{pmatrix}
u,
\quad
\delta
=
3^{1/4}5^{-1/8},
\quad
u>0.
\end{equation*}

With $u=1$, we have the vector $x_{2}^{\prime}\approx(1,0.9292,0.9292)^{T}$, which produces the order $\mathbf{A}\succ\mathbf{C}\equiv\mathbf{B}$.

The third column in the matrix presents another solution
\begin{equation*}
x_{2}^{\prime\prime}
=
\begin{pmatrix}
1
\\
1/\delta
\\
1
\end{pmatrix}
u,
\quad
\delta
=
3^{1/4}5^{-1/8},
\quad
u>0,
\end{equation*}
which for $u=1$ becomes $x_{2}^{\prime\prime}\approx(1,0.9292,1)^{T}$, and thus defines the order $\mathbf{A}\equiv\mathbf{C}\succ\mathbf{B}$.

By combining all least differentiating solutions, we put the schools in the order $\mathbf{A}\succeq\mathbf{C}\succeq\mathbf{B}$. Note that this result, as well as the most differentiating solution, which produces the order $\mathbf{A}\succ\mathbf{B}\succ\mathbf{C}$, significantly differ from that obtained by the traditional AHP in \cite{Saaty1977Scaling,Saaty1990Analytic} and given by $\mathbf{B}\succ\mathbf{A}\succ\mathbf{C}$.

\section{Conclusion and Discussion}

In the paper, we have developed a new approach to solve multi-criteria decision problems of ranking the priorities of choices from pairwise comparison judgments. The approach mainly follows the general AHP methodology, but offers a new analytical and computational framework based on tropical optimization to solve the problems in a different way. The new approach offers an exact direct solution to the problems in analytical form, and may have the potential to complement and supplement other AHP solutions.

The main differences between the proposed and traditional approaches are as follows. First, to approximate pairwise comparison matrices by consistent matrices, the new AHP applies rank-one matrix approximation in the log-Chebyshev sense instead of the approximation in Frobenius (or spectral) norm in the traditional AHP. Using the log-Chebyshev approximation yields the solution in the form of the tropical subeigenvectors of pairwise comparison matrices rather than the usual Perron vector of these matrices, provided by the traditional AHP. Note that the log-Chebyshev approximation is equivalent to minimizing the maximum relative error over the matrix entries. Therefore, this approximation technique seems to be quite reasonable to handle pairwise comparison matrices that consists of reciprocal entries with their values covering a wide range of magnitude. 

Furthermore, given the weights of criteria, the new AHP finds final priorities of choices by solving one optimization problem of the minimax weighted log-Chebyshev approximation rather than by obtaining separate solutions to the Frobenius approximation problems for each criteria and calculating the weighted sum of these solutions. The proposed minimax solution incorporates the weights into the objective function of the optimization problem, which provides a more general and comprehensive solution technique than that based on the direct calculation of the weighted sum. Specifically, this technique may result in a set of different solution vectors instead of a single solution in the traditional AHP, and thereby enhances the decision-making capabilities by extending the range of effective choice.

To simplify the analysis and interpretation of non-unique solutions, the whole set of priority vectors is characterized by the vectors, which most and least differentiate between alternatives with the highest and lowest priorities. The most and least differentiating vectors are found by maximizing and minimizing the Hilbert (span, range) seminorm of priority vectors in the solution set.

A key feature of the new approach is its close connection with tropical optimization, which results in a strong possibility to formulate all decision-making procedures as tropical optimization problems, and then to solve these problems directly using results available in the area of tropical optimization. In contrast to the traditional AHP, which involves numerical algorithms to calculate priority vectors, the application of tropical optimization yields analytical solutions, which describe all priority vectors in a compact vector form, ready for both formal analysis and immediate computations.  

Let us now consider the examples presented in the paper to discuss the difference between the outcomes of the traditional AHP and the tropical AHP, which we are suggesting here. Matrix $B$, from which the final vector of priorities results, is computed as the max-linear combination of the matrices $A_{i}$ for all criteria multiplied by the corresponding weight. In the vacation site selection example, all but two entries of this matrix come from $A_{3}$ (entertainment) and $A_{4}$ (way of travel), so all other criteria are not so important. Note that $A_{3}$ clearly ranks California better than short trips and both of them much better than the other two alternatives. The key entries of $A_{3}$ are those equal to $7$ and they ``survive'' (multiplied by some factors) in $B$ and in $(\mu^{-1}B)^{*}$. This is the main reason why the ranking of $A_{3}$ is also the final ranking, although the preference of California over Denver and Quebec is not so overwhelming as in $A_{3}$, due to some admixture from $A_{4}$ and other matrices. In contrast to this result, the traditional AHP ranks California third.

In the school selection example, six entries of $B$ come from $A_{1}$ (learning) and three entries come from $A_{4}$ (vocational training) with some help of $A_{6}$ (music classes). Friends, school life and college preparation are completely ruled out. This seems reasonable, since school life is unimportant, making friends is the same in all the schools and the ranking of $A_{5}$ (college preparation) is similar to that of learning but less important and less distinguishing between the schools. The second school is the champion in learning, but the first school is much better in vocational training and music lessons, which, in the end, makes it the winner, albeit with a small margin even for the most differentiating vector and a possibility of being on the par with the third school for one of the least differentiating vectors. The outcome of the traditional AHP is the same as if we judged the schools first with respect to learning and college preparation (which puts the second school first) and only then with respect to vocational training and music lessons (which decides between the remaining two schools).

Contemplating these examples and thinking of a general case, we see that the tropical AHP picks the highest entries of the matrices $A_{i}$ (with the corresponding weights) resulting in matrix $B$ on which the final comparison is based. Unlike the traditional AHP, the unimportant criteria and the criteria, which distinguish between the alternatives too weakly, are dispensed with. Also, a criterion $i$, which rates one alternative higher, can win over any number of other criteria which rate another alternative higher, if $A_{i}$ has high enough entries. After matrix $B$ is formed, we deal with a solution set to a Chebyshev approximation problem, in which we minimize the largest deviation between the logarithms of entries of $B$ and the logarithms of entries of a rank-one comparison matrix. From this solution set, we pick the solutions that most and least differentiate between the alternatives, and this allows us to see (unlike the vector produced by the traditional AHP) a whole set of reasonable priority vectors and rankings and, in particular, to observe when the judgments based on different criteria are in conflict with each other, and there is no clear winner among the alternatives. 

The school selection example shows that the new approach may lead to a set of weight vectors that represents the weights of criteria in parametric form as the whole tropical column span of the matrix $(\lambda^{-1}C)^{*}$. This involves the derivation of the most and least differentiating priority vectors based on solving parameterized maximization and optimization problems, which may be a difficult task. Another alternative would be to find the sets of the most, least and average (fair) differentiating vectors from the tropical column span of $(\lambda^{-1}C)^{*}$, but it is not clear whether this would necessarily lead to the most, least, and fair differentiating priority vectors in the end.   

Finally, we remark that it is possible to combine the traditional and tropical AHP approaches, for instance, by using the usual Perron eigenvector as the vector of criteria weights or by consistently using the geometric column barycenter of Kleene stars on both levels of the traditional AHP.

\section{Acknowledgements}

This work was supported by the Russian Foundation of Basic Research (RFBR) [grant number 18-010-00723] and the Engineering and Physical Sciences Research Council (EPSRC) [grant number EP/P019676/1].

The authors are grateful to the referees for valuable criticism of the initial versions of this paper, and, in particular, for suggesting to apply the tropical AHP to Saaty’s school selection example.



%% References


%\bibliography{Tropical_implementation_of_the_AHP_decision_method}






% \end{thebibliography}

\end{document}